\documentclass{amsart}
\usepackage{amssymb,amsfonts, amsthm, color,epsf}
\usepackage{enumerate}
\usepackage{sseq}
\usepackage{mathtools}
\usepackage{leftindex}
\usepackage{ifthen}

\usepackage{amsmath}
\usepackage{amscd}
\usepackage{amssymb}
\usepackage{amsthm}
\usepackage{xspace}
\usepackage[all,tips]{xy}
\usepackage{verbatim}
\usepackage{syntonly}
\usepackage{lineno}
\usepackage{todonotes}

\ifx\pdfpageheight\undefined
\usepackage[dvips,colorlinks=true,linkcolor=blue,citecolor=red,%
urlcolor=green]{hyperref}
\usepackage[dvips]{graphicx}
\makeatletter
\edef\Gin@extensions{\Gin@extensions,.mps}
\makeatother
\else
\usepackage[bookmarksopen=false,pdftex=true,breaklinks=true,%
backref=page,pagebackref=true,plainpages=false,%
hyperindex=true,pdfstartview=FitH,colorlinks=true,%
pdfpagelabels=true,colorlinks=true,linkcolor=blue,%
citecolor=red, urlcolor=darkgray, hypertexnames=false%
]%
{hyperref}
\fi

\usepackage{amsrefs}

\newtheorem{theorem}{Theorem}
\newtheorem{corollary}{Corollary}
\newtheorem*{theorem*}{Theorem}

\newtheorem{lemma}{Lemma}[subsection]
\newtheorem{proposition}[lemma]{Proposition}

\theoremstyle{definition}
\newtheorem{definition}{Definition}
\newtheorem{example}{Example}

\newtheorem{notation}{Notation}

\theoremstyle{remark}
\newtheorem{remark}{Remark}
\newtheorem{observation}{Observation}

\newcommand{\R}{\mathrm{R}}

\newcommand {\junk}[1]{}
\newcommand {\hide}[1]{}

\renewcommand {\prod} {\Pi}

\newcommand {\RR} {\mathrm{Reali}}
\newcommand {\card} {\mathrm{card}}

\newcommand{\Q}{\mathbb{Q}}

\newcommand{\Sh}{\mathbf{Sh}}

\newcommand {\HH}{\mathrm{H}}

\newcommand{\cH}{\mathcal{H}}
\newcommand{\eps}{\varepsilon}

\newcommand{\bbA}{{\mathbb A}}

\newcommand{\bbC}{{\mathbb C}}

\newcommand{\bbF}{{\mathbb F}}

\newcommand{\bbM}{{\mathbb M}}
\newcommand{\bbN}{{\mathbb N}}

\newcommand{\bbP}{{\mathbb P}}
\newcommand{\bbQ}{{\mathbb Q}}
\newcommand{\bbR}{{\mathbb R}}
\newcommand{\bbS}{{\mathbb S}}

\newcommand{\bbZ}{{\mathbb Z}}

\newcommand{\bT}{{\bf T}}

\newcommand{\bX}{{\bf X}}

\newcommand{\bZ}{{\bf Z}}

\newcommand{\frX}{{\mathfrak X}}

\newcommand{\frZ}{{\mathfrak Z}}

\newcommand{\cA}{{\mathcal A}}
\newcommand{\cB}{{\mathcal B}}
\newcommand{\cC}{{\mathcal C}}

\newcommand{\cF}{{\mathcal F}}
\newcommand{\cG}{{\mathcal G}}

\newcommand{\cL}{{\mathcal L}}

\newcommand{\cK}{{\mathcal K}}
\newcommand{\cM}{{\mathcal M}}
\newcommand{\cN}{{\mathcal N}}

\newcommand{\cS}{{\mathcal S}}

\newcommand{\cX}{{\mathcal X}}
\newcommand{\cY}{{\mathcal Y}}

\newcommand{\rD}{{\rm D}}
\newcommand{\rE}{{\rm E}}

\newcommand{\rH}{{\rm H}}

\newcommand{\rM}{{\rm M}}

\newcommand{\rR}{{\rm R}}

\newcommand{\Aut}{{\rm Aut}}

\newcommand{\Spec}{{\rm Spec}}

\newcommand{\isom}{\cong}

\newcommand{\vcd}{\mathrm{vcd}}
\renewcommand{\deg}{\mathrm{deg}}

\newcommand{\cosk}{\mathrm{cosk}}

\begin{document}
\title{Cohomological VC-density: Bounds and Applications}

\author{Saugata Basu}
\address{Department of Mathematics,
Purdue University, West Lafayette, IN 47906, U.S.A.}
\email{sbasu@math.purdue.edu}

\author{Deepam  Patel}
\address{Department of Mathematics, 
Purdue University, West Lafayette, IN 47906, U.S.A.}
\email{patel471@purdue.edu}

\date{\textbf{\today}}
\thanks{
Basu was  partially supported by NSF grant CCF-1910441. Patel was partially supported by a Simons Collaboration Grant.
}
\begin{abstract}
The concept of Vapnik-Chervonenkis (VC) density is pivotal across various mathematical fields, including discrete geometry, probability theory and model theory. In this paper, we introduce a topological generalization of VC-density.

Let $Y$ be a topological space and $\mathcal{X}$ a family of
closed subspaces of $Y$.
For each $p \geq 0$, 
we define a number, $\vcd^{p}_{\mathcal{X}}$,
which we refer to as the degree $p$ VC-density of the family $\mathcal{X}$.
The classical notion of VC-density within this topological framework can be recovered by setting  $p=0$
and restricted to singleton test families
\[
\mathcal{Z}=\{\{y\}:y\in Y\}.
\]
Our definition of degree $p$ VC-density extends to higher orders
as well.
For $p\geq 0, q \geq  1$, we define the degree $p$, order $q$ VC density
$\vcd^{p,q}_{\mathcal{X}}$ of $\mathcal{X}$, which recovers
Shelah’s notion of higher-order VC-density for $q$-dependent families  \cite{Shelah2014} when $p=0$
and restricted to singleton test families.
Our definition introduces a completely new notion when $p>0$.

We examine the properties of $\vcd_{\mathcal{X}}^p$ 
(as well as $\vcd^{p,q}_{\mathcal{X}}$)
when the families $\mathcal{X}$ are definable in structures with some underlying topology (for instance, 
the Euclidean topology for o-minimal structures over $\bbR$,
the analytic topology over $\bbC$, or  the \'{e}tale site for schemes over arbitrary algebraically closed fields). 
Our main result establishes that in any model of these theories 
\[
\vcd_{\mathcal{X}}^p \leq (p+1) \dim X,
\]
and more generally for any $q \geq 1$
\[
\vcd^{p,q}_{\mathcal{X}} \leq (p+q) \dim X.
\]
These results generalize known VC-density bounds in these structures \cites{RBG01, Basu10, Starchenko-et-al}, extending them in multiple ways, as well as providing a uniform proof paradigm applicable to all of them.
We give examples to show that our bounds are optimal in several cases. 

Our bounds 
actually go beyond model-theoretic contexts: they apply to 
arbitrary 
closed
correspondences of schemes with respect to singular, étale, or 
$\ell$-adic cohomology theories. 
We also present combinatorial applications of our higher-degree VC-density bounds, deriving higher degree topological analogs of well-known results such as the existence of 
$\eps$-nets and the fractional Helly theorem. 
 \end{abstract}
 
 \subjclass[2020]{Primary 03C52, 14F06; Secondary 14F20, 14F25}
\keywords{Constructible sheaves, schemes, cohomology, Vapnik-Chervonenkis density, NIP theories, o-minimal}
\maketitle

\tableofcontents

\section{Introduction}

\label{sec:intro}
Let $Y$ be a set and $\mathcal{X} \subset 2^Y$ a set of subsets of $Y$.
For any subset $Y'$ of $Y$, we 
set
\[
S(Y';\mathcal{X}) = \{Y' \cap \bX \mid \bX \in \mathcal{X}\}.
\]
One says that $\mathcal{X}$ \emph{shatters} $Y'$, if
$S(Y';\mathcal{X}) = 2^{Y'}$. In the following, we shall use the convenient notation  
\[
Y' \subset_n Y
\] 
to denote 
\[
Y' \subset Y \mbox{ and }  \card(Y') = n.
\]
In this setting, a famous result due to Sauer and Shelah asserts that:

\begin{theorem*}[Sauer-Shelah Lemma,\cites{Sauer, Shelah}] 
\label{lem:Sauer-Shelah}
Let
\begin{equation}
\label{eqn:nu}
\nu_{\mathcal{X}}(n) :=  \max_{Y' \subset_n Y} \card(S(Y';\mathcal{X})).
\end{equation}

Then, either for all $n > 0$,
        \[
        \nu_{\mathcal{X}}(n) = 2^n,
        \]
or there exists $c ,d  > 0$ such that
        \[
        \nu_{\mathcal{X}}(n) < c \cdot n^d,
        \]
for all $n>0$ (here $c,d$ are independent of $n$).
\end{theorem*}

The family $\mathcal{X}$ of sets is said to have finite \emph{Vapnik-Chervonenkis (henceforth VC)
dimension} if the second alternative is true in the above theorem. The \emph{Vapnik-Chervonenkis density of $\mathcal{X}$}, denoted by $\vcd_{\mathcal{X}}$, is defined by
\begin{equation}
    \label{eqn:def:vd-density}
    \vcd_{\mathcal{X}} = \limsup_n \frac{\log(\nu_{\mathcal{X}}(n))}{ \log(n)}.
\end{equation}

The Vapnik-Chervonenkis density is an important measure of `tameness'
or `complexity'  of set systems (see \cite{VC-book}). 
It plays an extremely important role in many different areas of mathematics including -- 
\begin{enumerate}
\item discrete and computational geometry, in proving the existence of $\eps$-nets of constant size (\cites{Haussler-Welzl,Komlos-Pach}), proving fractional Helly-type theorems (\cite{Matousek2004}), and studying point configurations over finite fields (\cite{Iosevich-et-al}),
\item in graph theory, related to the Erd\H{o}s-Hajnal property \cites{nguyen2024,CST2021,Basu10},
\item incidence combinatorics, 
for proving incidence bounds in distal structures \cite{Chernikov-et-al-2020}, bounds for the Zarankiewicz problem for semi-algebraic graph classes \cite{Pach-et-al} which includes point-hypersurface incidence problems, and bounding incidences over finite fields
\cite{Iosevich-et-al-incidence},
\item probability theory, for proving uniform convergence in certain situations (\cite{VC}),  

\item machine learning theory, in particular probably-approximately-correct (PAC) learning (\cite{Blumer-et-al}), 
\item model theory, where it is connected with the notion of the independence property of formulas (\cite{Simon}), 
\end{enumerate}
and a host of other applications. We refer the reader to \cite{Tenka-survey}
for a survey of some of these applications. 

Moreover, obtaining tight upper bounds on this density plays a central role in these applications and is considered to be an important problem. As a result, this problem has been considered from many different perspectives. We refer the reader to  \cites{Bartlett-et-al, McDonald-et-al, Starchenko-et-al, BP1} for some recent work on this topic.

While in the definition of VC-density it makes sense to consider arbitrary set systems $\mathcal{X} \subset 2^Y$, in model-theoretic applications one only considers definable families $\mathcal{X}$ of subsets of some definable set $Y$. More precisely, 
suppose $\bbM$ is an o-minimal expansion of $\mathbb{R}$,  
or a model of the theory of algebraically closed or real closed fields, and 
suppose that
$X,Y$ are definable subsets in 
$\bbM$,
and 
$H \subset X \times Y$ a definable subset. 

Then, 
\begin{eqnarray}
\label{eqn:def-of-X}
\mathcal{X} &:=& \{\leftindex_{x}{H} \mid x \in X\},  \\
\label{eqn:def-of-Y}
\mathcal{Y} &:=& \{H_y \mid y \in Y \}
\end{eqnarray}
are definable families of 
subsets of $Y$ and $X$ respectively,
where 
\[
\leftindex_{x}{H} = \pi_Y(\pi_X^{-1}(x) \cap H), 
H_y = \pi_X(\pi_Y^{-1}(y) \cap H),
\]
and 
$\pi_X: X \times Y \rightarrow X$, $\pi_Y: X \times Y \rightarrow Y$
are the projection maps. 

We say that $\mathcal{X}$ (respectively $\mathcal{Y}$) are closed definable families if $H$ is closed in $X \times Y$. In this case, each member 
$\leftindex_{x}{H}$ (respectively $H_y$) is a closed definable subset of 
$Y$ (respectively $X$).

\begin{remark}[Terminology]
We sometimes refer to $\mathcal{X}$ as the \emph{candidate family}.
\end{remark}

Below we shall sometimes refer to $\vcd_{\mathcal{X}}$ as the \emph{VC-codensity} of 
$\mathcal{Y}$. 

A formula $\phi(x;y)$ has the independence property in the theory  $\bT$, if for each $n >0$, $\bT \vdash I_n$,
where 
\[
I_n := (\exists x_0) \cdots (\exists x_{n-1}) (\exists y_\emptyset)\cdots 
(\exists y_W) \cdots (\exists y_{n})
\bigwedge_{W \subset n}
\bigl(
\bigwedge_{i \in W} \phi(x_i,y_W) \wedge \bigwedge_{i \not\in W} \neg\phi(x_i,y_W)
\bigr),
\]
and $n$ denotes $\{0,\ldots,n-1\}$.
In other words, in any model  $\bbM$ of $\bT$, and for every $n >0$, the definable family in $\bbM^{|x|}$ defined by $\phi(x;y)$ shatters
some subset of $\bbM^{|x|}$ of cardinality $n$.

We say that the theory $\bT$ has the independence property if some formula
$\phi(x;y)$ has the independence property. We say that $\bT$ is an NIP
theory if $\bT$ does not have the independence property. Using the Sauer-Shelah Lemma, NIP
theories are precisely those for which
$\vcd_{\mathcal{X}} < \infty$ for all definable families.

The NIP property is of fundamental importance in model theory (see \cite{Simon}). Many
well known theories have the NIP property 
including:
\begin{enumerate}[(a)]
    \item 
    \label{itemlabel:a}
    theory of algebraically closed fields of characteristic $0$ (denoted $\mathrm{ACF}(0)$), 
    \item 
    \label{itemlabel:b}
    theory of algebraically closed fields of characteristic $p >0$ (denoted $\mathrm{ACF}(p)$), 
    \item 
    \label{itemlabel:c}
    the theory of real closed fields (denoted $\mathrm{RCF}$), 
    \item 
    \label{itemlabel:c'}
    theories of o-minimal expansions of  $\mathbb{R}$,
    \item 
    \label{itemlabel:d}the theory of algebraically closed valued fields with characteristics pair $(p,p'$) (denoted $\mathrm{ACVF}(p,p')$).
\end{enumerate}

In each of the NIP theories 
listed above
and assuming that $X \subset \bbM^k$ (see \eqref{eqn:def-of-X}), one has a bound
\begin{equation}
\label{eqn:vcd-bound}
\vcd_{\mathcal{X}} \leq k.
\end{equation}

The proofs of inequality \eqref{eqn:vcd-bound} for the different 
theories listed above appears in \cite{RBG01} for theories
\eqref{itemlabel:a} and \eqref{itemlabel:b}, in
\cite{Basu10} for \eqref{itemlabel:c} and \eqref{itemlabel:c'}, and 
\cite{BP1} for \eqref{itemlabel:d}.

In order to bound $\vcd_{\mathcal{X}}$,
it is often useful to consider the dual family $\mathcal{Y}$ of subsets of $X$, and bound the number of {\it realizable $0/1$-patterns} for any 
$n$ members of this family. 

Given any set $S$, and $\mathcal{F} \subset_n 2^S$,
a $0/1$-pattern
on $\mathcal{F}$ is an element of $\{0,1\}^{\mathcal{F}}$. We say that a $0/1$-pattern $\sigma \in \{0,1\}^{\mathcal{F}}$ is realizable if and only if $\exists s \in S, \chi_{F}(s) = \sigma(F)$ for all $F \in \mathcal{F}$, where we denote by $\chi_F$ the characteristic function of $F$. Denote by $\hat{\nu}_{\mathcal{Y}}(n)$ the maximum number of realizable 
$0/1$-patterns where the maximum is taken over all finite subsets
$\mathcal{Y}' \subset_n \mathcal{Y}$. It is an easy exercise to check that
\[
\nu_{\mathcal{X}}(n) = \hat{\nu}_{\mathcal{Y}}(n).
\]
In particular, the VC-codensity of $\mathcal{Y}$ (or equivalently the VC-density of $\mathcal{X}$)
is equal to 
\[
\limsup_n \frac{\log \hat{\nu}_\mathcal{Y}(n)}{\log n}.
\]

Inequality \eqref{eqn:vcd-bound} (and its dual version which gives
bounds on realizable $0/1$-patterns)
plays a fundamental role in many applications in combinatorics, discrete geometry as well as in theoretical computer science. 
For example, 
Ronyai, Babai and Ganapathy in \cite{RBG01} proved 
a tight bound on the number of $0/1$-patterns of finite sets of polynomials
over an arbitrary field, 
from which one can deduce immediately
that the VC-density of the family 
\[
\left(\{f \in k[X_1,\ldots,X_m]_{\leq d} \mid f(x) = 0 \}\right)_{x \in k^m}
\]
is at most $m$.
They utilized their bound on the number of $0/1$-patterns
to prove tight bounds on projective dimensions of graphs, and of span programs. 
In the setting of ordered fields, 
similar bounds on the number of sign conditions (instead of $0/1$-patterns) play a key role
in obtaining upper bounds on the speed of semi-algebraically defined graph classes \cite{Sauermann}. 
Bounds of VC-density enter into the bounds on sizes of epsilon nets \cites{Haussler-Welzl,Komlos-Pach}, 
the fractional Helly number \cite{Matousek2004}, in recent work on
bounding incidences \cites{Iosevich-et-al-incidence, Chernikov-et-al-2020}, 
in machine-learning 
\cite{Blumer-et-al} and in many other applications.

Suppose that $\bbM$ is a model of any of the theories listed before. 
The definition of VC-density is related to cardinalities of finite 
subsets of $Y = \bbM^d$.  From the perspective of applications in discrete geometry and incidence combinatorics, it is very natural to replace finite point sets by finite \emph{arrangements of  sets} belonging to some fixed definable family. 
For example, taking $\bbM = \bbR$, 
the combinatorial study of finite sets of points of $\bbR^d$ (for some $d > 0$)  is equivalent using (projective or affine) duality to studying finite hyperplane arrangements. Combinatorial properties of real hyperplane arrangements in affine and projective spaces have been studied in great depth (see \cite{Pach-Sharir}*{Chapter 2}). The next step has been to study the complexity of arrangements of more general subsets. For example, real sub-varieties or semi-algebraic sets defined by polynomials of degrees bounded by some constants (see \cite{Agarwal-Sharir}).

Another instructive example of a generalization of 
combinatorial results from points to higher dimensional arrangements is the polynomial partitioning theorem originally proved by Guth and Katz for finite sets of points in $\bbR^d$ \cite{Guth-Katz} and which plays an extremely important role in many recent works on incidence combinatorics and harmonic analysis. This result was later generalized by Guth \cite{Guth} to a polynomial 
partition theorem for finite sets of real varieties rather than points, and this generalized partition theorem has found 
new applications (see for example \cite{Aronov-et-al}).

A third example of this nature in discrete geometry 
is the generalization of the classical
Helly's theorem for convex sets (which can be viewed as a theorem about point transversals), to the case of higher dimensional transversals
where points are replaced by hyperplanes due to Goodman and Pollack \cite{GP}
(see also \cite{GPW}).

It is thus natural to try to generalize the notion of VC-density
defined in \eqref{eqn:def:vd-density} to a setting where in the 
definition of 
$\nu_{\mathcal{X}}(n)$ in \eqref{eqn:nu}, the 
elements of the sets $Y'$ are subsets
(belonging to some definable family) rather than points of $Y$
(in the same spirit as polynomial partitioning theorem of varieties
\cite{Guth} generalizes that of points in \cite{Guth-Katz}).
We could then aim to generalize the fundamental inequality \eqref{eqn:vcd-bound} in this setting. Such a generalization would be a natural extension (in view of the prior developments in the theory of arrangements discussed in the previous paragraph), and we believe will open up a new dimension of combinatorial problems in the study of arrangements.

Motivated by applications to higher dimensional analogs of the aforementioned applications of VC-density, and in particular, the considerations of the previous paragraph, the main goal of this article is to define a higher dimensional analog of VC-density, prove tight bounds for this higher VC-density, and provide some applications to  discrete geometry. 

As a point of departure, we note that the proofs of inequality \eqref{eqn:vcd-bound} in different theories are quite different and some of these proofs have a {\it topological} flavor. Indeed, there is often a natural topology (or more generally a Grothendieck topology) on the definable sets (for example, the Zariski or \'{e}tale sites in the case of $\mathrm{ACF}$ , the euclidean topology in the case of RCF and o-minimal theories). In this article, we leverage this additional topological structure in such settings in order to define higher degree notion of VC-density. Our generalization of VC-density will be indexed by a bidegree $\vcd^{p,q}$ (with $p,q \in \bbN, p \geq 0, q \geq 1$) and so that $\vcd^{0,1}$ recovers the classical notion discussed above 
(using singleton test families).

\subsection{Cohomological VC-density}
The first obstacle one meets in attempting to generalize the notion of VC-density is observed in providing an interpretation of $\card(S(Y';\mathcal{X}))$ in the setting where elements of $Y'$ are allowed to be definable subsets of $Y$. Below we provide an interpretation of this object for such $Y'$ in the setting where our structures are equipped with a topology. 

\begin{remark}
In fact, we do not require that our structures are equipped with a topology. Rather, the construction below works in any structure equipped with a Grothendieck topology. For example, our results are applicable in the $\mathrm{ACF}$  setting, with the underlying Grothendieck topology given by the \'{e}tale topology (as discussed in the text).
\end{remark}

Suppose now that $Y$ is a topological space and $\mathcal{X}$ is a set of \emph{closed} subspaces of $Y$. Below, we denote by $\rH_i(Y,\bbQ)$ (respectively $\rH^i(Y,\bbQ)$) the usual singular homology (respectively cohomology) of $Y$
 with rational coefficients. Given any 
 (discrete)
 finite subset $Y' \subset Y$, the different intersections
$Y' \cap \bX, \bX \in \mathcal{X}$ are each characterized by the image
of the linear map $$\HH_0(Y' \cap \bX,\bbQ) \rightarrow \HH_0(Y',\bbQ).$$
Recall that for a discrete set of points $P$, the corresponding $\HH_0(P,\bbQ)$ is a $\bbQ$-vector space with basis given by $P$. The linear map above is then the natural inclusion, and the image subspace characterizes the set $Y' \cap \bX$.

In what follows it will be more convenient to work with cohomology rather than homology, and the corresponding cohomological statement is that $Y' \cap \bX, \bX \in \mathcal{X}$ are each characterized by the kernels of the linear map $\HH^0(Y',\bbQ) \rightarrow \HH^0(Y' \cap \bX,\bbQ)$ (induced by restriction). More precisely, for any subspace $Y' \subset Y$,
we denote by $\HH^i(Y',\bbQ)$ the sheaf cohomology
of $Y'$ with values in the constant sheaf $\Q$.\footnote{For good spaces, these are the usual singular cohomology groups.} Then, for
any $Y' \subset_n Y$, and $\bX \in \mathcal{X}$,
$Y' \cap \bX$ is determined by the kernel of the homomorphism
\[
\HH^0(Y',\bbQ) \rightarrow \HH^0(Y' \cap \bX, \bbQ),
\]
\begin{observation}
With notation  as above, let 
\[
\bbS^0(Y';\mathcal{X}) = \{\ker(\HH^0(Y',\bbQ) \rightarrow \HH^0(Y' \cap \bX, \bbQ)) \mid \bX \in \mathcal{X}\}.
\]

It follows that
\[
\card(S(Y';\mathcal{X})) = \card(\bbS^0(Y';\mathcal{X})).
\]
\end{observation}
This is the observation that is at the heart of our generalization of VC-density to higher degrees and suggests a natural generalization of the notion of VC-density to higher (cohomological) degrees.

We shall now state our notion of higher degree VC-density. 
We first define our notion in a general topological context
and then specialize it to the more restrictive model theoretic situations.

Let $Y$ be a topological space and
$\mathcal{X}, \mathcal{Z} \subset 2^Y$ be sets of \emph{closed}  subspaces of $Y$. Let
$\mathcal{Z}_0 \subset_n \mathcal{Z}$,
and
let $\bigcup \mathcal{Z}_0$ denote  $\bigcup_{\bZ \in \mathcal{Z}_0} \bZ$. For each $p \geq 0$, define
\[
\bbS^p(\mathcal{Z}_0;\mathcal{X}) = \{\ker(\HH^p(\bigcup \mathcal{Z}_0,\bbQ) \rightarrow \HH^p(\bigcup \mathcal{Z}_0 \cap \bX, \bbQ)) \mid \bX \in \mathcal{X}\}.
\]

\begin{remark}[Terminology]
We sometimes refer to $\mathcal{Z}$ as the \emph{test family}.
If every member of $\mathcal{Z}$ is a singleton subset of $Y$,
we call $\mathcal{Z}$ a \emph{singleton test family}.
\end{remark}

We define the \emph{degree-$p$ VC-density of 
$\mathcal{X},\mathcal{Z}$} by:

\begin{equation}
\label{eqn:def:vcd:p'}
\vcd^p_{\mathcal{X},\mathcal{Z}} := 
\limsup_{n}  \frac{ \log(\nu_{\mathcal{X},\mathcal{Z}}^p(n))}{\log(n)},
\end{equation}
where
\begin{equation}
\label{eqn:chi-C:p}
\nu^p_{\mathcal{X},\mathcal{Z}} (n) =  \sup_{\mathcal{Z}_0 \subset_n \mathcal{Z}} \card(\bbS^p(\mathcal{Z}_0;\mathcal{X})).
\end{equation}

\hide{
In any of the definable contexts that we consider in this paper, 
namely in models of the
theory of an o-minimal expansion of $\mathbb{R}$, $\mathrm{RCF}$,
$\mathrm{ACF}(0)$ or  $\mathrm{ACF}(p)$, 
we restrict to definable sets $Y$, and definable families $\mathcal{X},\mathcal{Z}$ of closed subspaces of $Y$. We define:

\begin{equation}
\label{eqn:def:vcd:p}
\vcd^p_{\mathcal{X}} := 
\sup_{\mathcal{Z}} 
\limsup_{n}  \frac{ \log(\nu_{\mathcal{X},\mathcal{Z}}^p(n))}{\log(n)},
\end{equation}
where the 
supremum
in \eqref{eqn:def:vcd:p}
is taken over all \emph{definable} families $\mathcal{Z}$ of closed subspaces  of $Y$. Notice that unlike in \eqref{eqn:def:vcd:p'}, the 
definition of $\vcd_{\mathcal{X}}^p$ in \eqref{eqn:def:vcd:p} depends only on one definable family $\mathcal{X}$ of closed subsets of $Y$.  
}
In any of the definable contexts that we consider in this paper, 
namely in 
o-minimal expansions of $\mathbb{R}$, or
models of the
theory 
$\mathrm{RCF}$,
$\mathrm{ACF}(0)$ or  $\mathrm{ACF}(p)$, 
we restrict to definable sets $Y$, and uniformly definable families
$\mathcal{X},\mathcal{Z}$ of closed subspaces of $Y$.
In real closed models, ``proper'' means definably compact and
cohomology means cohomology on the o-minimal site with constant
$\bbQ$ coefficients; over $\mathbb R$ this agrees with the usual
cohomology of definable sets.  Over an arbitrary algebraically
closed field $k$, we use \'{e}tale cohomology with a fixed
coefficient field $\bbQ_\ell$ or $\bbF_\ell$, where
$\ell\neq\operatorname{char}(k)$, including when
$\operatorname{char}(k)=0$.  Over $\bbC$, rational singular
cohomology is also allowed.  The symbol $\bbQ$ in the general
formulas is interpreted according to this convention.

For the cohomological growth functions, we take the supremum
in $[1,+\infty]$, so that its value is $1$ when there are no
admissible test subfamilies.  In particular, a finite test family
has density zero.  We define:

\begin{equation}
\label{eqn:def:vcd:p}
\vcd^p_{\mathcal{X}} := 
\sup_{\mathcal{Z}} 
\limsup_{n}  \frac{ \log(\nu_{\mathcal{X},\mathcal{Z}}^p(n))}{\log(n)},
\end{equation}
where the 
supremum
in \eqref{eqn:def:vcd:p}
is taken over all closed \emph{definable} families $\mathcal{Z}$ of 
subspaces  of $Y$. Notice that unlike in \eqref{eqn:def:vcd:p'}, the 
definition of $\vcd_{\mathcal{X}}^p$ in \eqref{eqn:def:vcd:p} depends only on one definable family $\mathcal{X}$ of closed subsets of $Y$.  
One key observation is the following.

In 
an o-minimal expansion of $\mathbb{R}$, or any model of the theories 
$\mathrm{RCF}$,
$\mathrm{ACF}(0)$ or $\mathrm{ACF}(p)$,
with $\mathcal{X},\mathcal{Z}$ definable families of closed definable 
subsets of a proper definable set $Y$,
\begin{equation}
\label{eqn:finite}
\card(\bbS^p(\mathcal{Z}_0;\mathcal{X})) < \infty,
\end{equation}
for all finite subsets $\mathcal{Z}_0 \subset \mathcal{Z}$.

Note that the finiteness of $\bbS^p(\mathcal{Z}_0;\mathcal{X})$ does not follow from standard
model-theoretic arguments but relies on the properties of the 
underlying topology. Indeed it is a special case of the results proved in this paper 
(Theorem~\ref{thm:main} below).

The main result of our paper is the following theorem.

\begin{theorem}
\label{thm:informal}
Let $\mathbf{T}$ be 
one of the theories:  
$\mathrm{RCF}$,
$\mathrm{ACF}(0)$ or $\mathrm{ACF}(p)$.
For every definable family $\mathcal{X} = \{\leftindex_{x}{H} \mid x \in X\}$,
with $H$ a closed definable subset of $X \times Y$,
of 
definable subsets of some proper definable set $Y$ in 
an o-minimal expansion of $\mathbb{R}$, or
any model of $\mathbf{T}$, and $p \geq 0$,
\[
\vcd^p_{
\mathcal{X}
} 
\leq (p+1) \cdot \dim X.
\]
\end{theorem}

Inequality \eqref{eqn:vcd-bound} (in the theories mentioned in Theorem~\ref{thm:informal}) is
then recovered as a special case with $p=0$, noting that
\[
\vcd_{\mathcal{X}} = \limsup_{n}  \frac{ \log(\nu_{\mathcal{X},\mathcal{Z}}^0(n))}{\log(n)}
\leq 
\vcd_{\mathcal{X}}^0,
\]
with
\begin{equation}
\label{eqn:Y}
 \mathcal{Z} = \{\{y\} \mid y \in Y\}.
\end{equation}

\begin{remark}
    The restriction of properness of $Y$ and that of the members of the family $\mathcal{X}$ in Theorem~\ref{thm:informal} is not very important in the derivation of \eqref{eqn:vcd-bound} for general definable families $\mathcal{X}$ -- since  inequality \eqref{eqn:vcd-bound} for general $\mathcal{X}$ follow quite easily from the same inequality in the proper case (using a reduction similar to the one used in the proof of Theorem 1 in \cite{BP1}).   
\end{remark}

Note that our notion of 
$\vcd^0_{\mathcal{X}}$
is in fact 
more general than the classical VC-density $\vcd_{\mathcal{X}}$, 
since the supremum in \eqref{eqn:def:vcd:p} is taken over
arbitrary definable families of proper definable subsets of $Y$, 
rather than the one in \eqref{eqn:Y}.
In particular, even in the classical case $p=0$, our theorem generalizes existing results in the literature.

While our main theorem, Theorem~\ref{thm:informal}, only depends on 
a suitably chosen cohomology theory in the models of the theories that 
we consider, our proof of the main technical result (Theorem~\ref{thm:main} below) from which Theorem~\ref{thm:informal} 
follows directly, uses properties of the category of {\it constructible sheaves} which underlie these cohomology theories. As a result we find it advantageous to shift to a more geometric language and prove our technical results in three different categories (see Table~\ref{tab:versions} below).
The properties of constructible sheaves that we need hold in these categories. 
Moreover, the technical result (Theorem~\ref{thm:main}) that we state below is valid in all the three  categories listed above and allows us to derive Theorem~\ref{thm:informal} for 
an o-minimal expansion of $\mathbb{R}$, as well as all models of 
$\mathrm{RCF}$ (from the first category), 
models of 
$\mathrm{ACF}(0)$ (from the second category in the case of $\bbC$ and the third category for arbitrary algebraically closed fields of characteristic $0$), and 
models of 
$\mathrm{ACF}(p)$ from the third category.

Since in our proofs we will draw on techniques from algebraic geometry,
 we prefer to use the notion of `correspondence' in lieu of definable families.
The notion of correspondence gives a more geometric
and scheme-theoretic way to encode `definable families' in model theory.

\begin{definition}[Correspondence]
\label{def:correspondence}
Let $A$, $B$ be schemes. A {\it correspondence} between $A$ and $B$ is a scheme $D$ equipped with morphisms $\pi_A: D \rightarrow A$ and $\pi_B: D\rightarrow B$. A correspondence is {\it finite}  if the resulting morphism $D \rightarrow A \times B$ is a finite map. We shall use the notation $[D;A,B]$ to denote a correspondence. In particular, a closed immersion $D \hookrightarrow A \times B$ gives a finite correspondence.
\end{definition}

In this article, all correspondences will in fact be closed immersions (i.e. the map $D \rightarrow A \times B$ will be a closed immersion).

\begin{remark}
Definition~\ref{def:correspondence} stated above in the category of schemes over complex numbers, extends in an obvious way to the category of schemes over other algebraically closed fields. In the 
category of definable sets and maps in any o-minimal structure, by a correspondence we shall always mean a closed immersion. We 
do not remark further on this point in what follows.
\end{remark}

\subsection{Key technical result}
\label{subsec:main}
In this section we state our key technical result (Theorem~\ref{thm:main} below): 
namely a common upper bound on the function $\nu^{p}_{\mathcal{X},\mathcal{Z}}$ in each of the theories referred to in Theorem~\ref{thm:informal}. 
Theorem~\ref{thm:informal} follows immediately from this bound. 

In order to prove our upper bound we work in the following 
three different categories:
\begin{enumerate}[1.]
\item
the category whose objects and maps are definable in some fixed o-minimal expansion of $\mathbb{R}$;
\item 
the category of schemes of finite type over $\mathbb{C}$;
\item 
the category of schemes of finite type over an arbitrary algebraically closed field;
\end{enumerate}

The families $\mathcal{X}$ and $\mathcal{Z}$, and the
correspondences $[H;X,Y]$ and $[\Lambda;Y,Z]$ that defines them,
as well as the cohomology theory being used, have to be interpreted differently in the three  different categories
listed above. We make these explicit below.
For the benefit of the readers and easy reference 
we also summarized this information in Table~\ref{tab:versions}.


\begin{table}
    \centering
    \begin{tabular}{|p{2.5cm}|p{2.5cm}|p{2.5cm}|p{2.5cm}|}
    \hline
        Category & Correspondences \newline
        $[H;X,Y]$ and $[\Lambda;Y,Z]$ & Topology/Site& Cohomology theory \\
    \hline
  o-minimal expansion 
  of $\mathbb{R}$      & 
  $X,Y,Z$ definable sets; \newline  
  $Y$ compact; 
  \newline
  $H \subset X \times Y$ and $\Lambda \subset Y \times Z$ closed definable subsets & euclidean & 
  sheaf cohomology groups of the constant sheaf $\bbQ$ in the underlying euclidean topology
  \\
  \hline
 schemes of finite type
 over $\mathbb{C}$ & $X,Y,Z$ schemes over $\mathbb{C}$;
 \newline 
 $Y$ proper; 
 \newline
 $H \subset X \times Y$, $\Lambda \subset Y \times Z$ closed subschemes &  analytic topology & sheaf cohomology groups of the constant sheaf $\bbQ$  in the underlying complex analytic topology \\
 \hline
schemes of finite type over an algebraically closed field $k$ \newline
with $\mathrm{char}(k) = p$      
& $X,Y,Z$ schemes of finite type over $k$;
\newline $Y$ proper; 
\newline
$H \subset X \times Y$ and $\Lambda \subset Y \times Z$ closed subschemes & \'{e}tale site &
$\ell$-adic cohomology  
defined on the \'{e}tale
site with $\ell \neq p$. \\
\hline

\end{tabular}
\vspace{.2in}
    \caption{The three categories}
    \label{tab:versions}
\end{table}


\subsubsection{o-minimal case}
\label{subsubsec:om}
We start with the o-minimal case.
\paragraph{\emph{Category}.} We fix an o-minimal expansion of 
$\mathbb{R}$ and restrict to the category of
definable sets and maps in this structure.

\paragraph{\emph{Correspondences $[H;X,Y]$ and $[\Lambda;Y,Z]$}.}
Let $X,Y,Z$ be definable sets, $Y$ compact,
and $H \subset X \times Y$ and $\Lambda \subset Y \times Z$ be closed definable subsets, and let $[H;X,Y]$ and $[\Lambda;Y,Z]$
be the induced correspondences.

\paragraph{\emph{Cohomology}.}
We use the euclidean topology on 
definable subsets and define the cohomology of a definable subset to be that of the constant sheaf $\mathbb{Q}$. 

\hide{
\begin{remark}
For more general o-minimal structures (namely,
o-minimal expansions of arbitrary real closed fields), one needs to be
more careful about defining a good cohomology theory 
(see for example \cites{Edmundo-Woerheide,Woerheide})
and the same remark as in Remark~\ref{rem:site} applies in this situation
as well.
\end{remark}
}

\subsubsection{Complex case}
\label{subsubsec:complex}
\paragraph{\emph{Category}.}
In this case we restrict to the category of schemes of finite type over the complex numbers. 

\paragraph{\emph{Correspondences $[H;X,Y]$ and $[\Lambda;Y,Z]$}.}
Let $X,Y,Z$ be schemes over the field of complex numbers. We assume that $Y$ is a proper scheme. Let $H \subset X \times Y$,  and $\Lambda \subset Y \times Z$ be closed subschemes. In particular, we are given finite correspondences
$[H;X,Y]$ and $[\Lambda;Y,Z]$.\\

\paragraph{\emph{Cohomology}.} 
The cohomology groups that we will use are the sheaf cohomology groups of the constant sheaf $\bbQ$ in the underlying complex analytic topology. 
Note that these are 
isomorphic to the ordinary singular cohomology with coefficients in $\bbQ$,
with isomorphisms that are functorial.

\begin{remark}
\label{rem:site}
For more general algebraically closed fields of characteristic zero,
one can take the cohomology of the constant sheaf on a certain Grothendieck site (defined using semi-algebraic triangulations) 
instead of the analytic topology, and these groups will be canonically isomorphic to the ones defined using analytic topology in the 
complex case.
\end{remark}

\subsubsection{\'{E}tale case}
\label{subsubsec:etale}

\paragraph{\emph{Category}.}
In this case we restrict to the category of schemes defined over an algebraically closed field $k$ (of arbitrary characteristic).

\paragraph{\emph{Correspondences $[H;X,Y]$ and $[\Lambda;Y,Z]$}.}
Let $X,Y,Z$ be schemes over an algebraically closed field $k$ of $char(k) = p \geq 0$, and fix a prime $\ell \neq p$. We assume that $Y$ is a proper scheme. Let $H \subset X \times Y$ and $\Lambda \subset Y \times Z$ be closed subschemes. 
In particular, we are given finite correspondences
$[H;X,Y]$ and $[\Lambda;Y,Z]$.\\

\paragraph{\emph{Cohomology}.} 
We consider in this case the $\ell$-adic cohomology defined on the \'{e}tale
site. In particular, given a scheme $X$ over $\Spec(k)$, we denote by $\rH^i(X,\bbF_{\ell})$ the \'{e}tale cohomology of $X$ with coefficients in the finite field $\bbF_{\ell}$, and by $\rH^i(X,\bbQ_{\ell})$ the $\ell$-adic \'{e}tale cohomology with coefficients in $\bbQ_{\ell}$.

\subsubsection{The families $\mathcal{X}$ and $\mathcal{Z}$.}
In each of the three categories of Sections~\ref{subsubsec:om}, \ref{subsubsec:complex} and \ref{subsubsec:etale},
let $\pi_{H,X}: H \rightarrow X, \pi_{H,Y}: H \rightarrow Y, \pi_{\Lambda,Y}: \Lambda \rightarrow Y,$ and $ \pi_{\Lambda,Z}:\Lambda \rightarrow Z$ denote the natural projection maps. For a closed point $z \in Z$, we denote by $\Lambda_z \subset Y$ the closed
subspace $\pi_{\Lambda,Y}(\pi_{\Lambda,Z}^{-1}(z))$. Similarly, for
a closed point $x \in X$, we denote by $\leftindex_{x}{H} \subset Y$ the closed subspace $\pi_{H,Y}(\pi_{H,X}^{-1}(x))$. Finally, we denote by $\leftindex_{x}{\Lambda}_{z} = \Lambda_z \cap \leftindex_x{H}$.

\subsubsection{Upper bound on $\nu^p_{\mathcal{X},\mathcal{Z}}$}
\label{subsubsec:main}
In each of the three categories (see Table~\ref{tab:versions}) we have the following theorem. Later we will refer to the three versions as o-minimal, complex and \'{e}tale versions of Theorem~\ref{thm:main}.

\begin{theorem}\label{thm:main}    
There exists a constant $C > 0$ 
(depending only on the correspondences $[H;X,Y]$ and $[\Lambda;Y,Z]$ and $p$),
such that for all $n \geq 1$,
\[
\nu^p_{\mathcal{X},\mathcal{Z}}(n) \leq C \cdot n^{(p+1)\dim X},
\]
where $\mathcal{X} = \{\leftindex_{x} H \mid x \in X\}$ and 
$\mathcal{Z} = \{\Lambda_z \mid z \in Z\}$.
\end{theorem}

\subsection{Method of Proof}
\label{subsec:method}
We briefly outline the proof method for  the
o-minimal, complex and \'{e}tale versions of Theorem~\ref{thm:main}.
Since the proofs are structurally quite similar in each of the three cases
we can describe them in a uniform way.
For the purposes of this subsection
we shall simply refer to our objects as spaces with the understanding that one is in one of the 
aforementioned settings. In particular, we work with data $[H;X,Y], [\Lambda;Y,Z]$ in these settings. Furthermore, we simply work with singular cohomology with rational coefficients, again with the understanding that 
over arbitrary algebraically closed fields (i.e. not necessarily over the complex numbers) one should replace these with appropriate $\ell$-adic \'{e}tale cohomology groups.  

Our first observation is that for every choice of 
$z_0,\ldots,z_n$, 
the kernels 
\begin{equation}\label{equation:mapincohomology}
    \ker\left(
\mathrm{H}^p\left(\bigcup_{j=0}^{n} \Lambda_{z_{j}},\bbQ\right)
    \longrightarrow
    \mathrm{H}^p\left(\bigcup_{j=0}^{n} \leftindex_{x}{\Lambda}_{z_{j}},\bbQ\right)
    \right)
\end{equation}
are the stalks of a constructible sheaf on $X$ which is a subsheaf 
of a constant sheaf. Indeed, each of the cohomology groups can be realized as the stalk at the point $x \in X$ of a constructible sheaf on $X$. The cohomology group on the left is in fact the stalk of a constant sheaf (which reflects the fact that this group does not depend on $x$). Using general results on constructible sheaves (see \ref{prop:kernelisindep}, \ref{subsubsec:o-minimalsheaves},\ref{subsec:constructibleetale}), we conclude that in order to prove the theorem it suffices to prove that there exists a 
definable partition of $X$ of size bounded by $C\cdot n^{(p+1)\dim X}$ such that this constructible sheaf (i.e. the kernel of the aforementioned constructible sheaves)
is constant when restricted to each part of this partition. 

We prove the existence of this partition 
by proving that there
exists a family $\mathcal{S}$ of (definable) subsets of $X$ (depending only on the given correspondences $[H;X,Y]$ and $[\Lambda;Y,Z]$) parametrized by the spaces $Z, Z^2,\ldots,Z^{p+1}$ and satisfying the following two properties:
\begin{enumerate}
\item For every choice of $z_0,\ldots,z_n \in Z$, one obtains 
$\sum_{j=0}^{p} C \cdot \binom{n+1}{j+1}$ closed (definable) subspaces of $X$.

\item The closed subspaces from the previous part are such that on each connected component of the realizations of any $0/1$-pattern on this set of 
subspaces the aforementioned constructible sheaf (i.e. the sheaf whose stalks are the kernels appearing above) is constant. 
\end{enumerate}

The bound in the theorem then follows from a bound on the total number of connected components of the realizations of $0/1$-patterns of these $\sum_{j=0}^{p} C \cdot \binom{n+1}{j+1}$ closed subspaces of $X$. 
In order to bound the latter quantity in the complex scheme setting, 
we use Theorem~\ref{thm:Hilbert} which is itself a consequence of 
\cite[Theorem 2.1]{Basu-Parida}.
In the o-minimal setting, we use 
Theorem~\ref{thm:vcdensityapp:om} which is a consequence of
\cite[Theorem 2.2]{Basu10}.

In order to prove the existence of the family $\mathcal{S}$ of $X$, we first consider the simplicial space $\frX$ with $k$-simplices given by $ X \times Z^{k+1}$. We associate to our correspondences $H$ and $\Lambda$ the following data:
\begin{enumerate}
\item Simplicial spaces $\frX_{H,\Lambda}$ (respectively $\frX_{\Lambda}$) depending on $[H;X,Y]$ (respectively $X$ and $[\Lambda; Y, Z]$).
\item These simplicial spaces fit into a commutative diagram
$$
\xymatrix{
\frX_{H,\Lambda} \ar[r]^{\pi_{\bullet}} \ar[dr]^{f_{\bullet}} & \frX_{\Lambda} \ar[d]^{g_{\bullet}} \\
   & \frX
}
$$
\end{enumerate}

The proof now proceeds as follows:
\begin{enumerate}
\item 
\label{itemlabel:step1}
The simplicial map $f_{\bullet}$ gives morphisms $f_k: \frX_{H,\Lambda,k} \rightarrow X \times Z^{k+1}$ at the level of simplices. We obtain the collection of subspaces $\mathcal{S}$ discussed above from fixed stratifications
$\cS_k$ such that $Rf_{k,*}\bbQ$ is locally constant on the strata of $\cS_k$ for all $0 \leq k\leq p$. Note that this data is independent of $n$. 
\item The aforementioned simplicial spaces are parametrized by the simplicial space $\frZ$ with $k$-simplices given by $Z^{k+1}$. In particular, we may restrict these spaces over a point $(z_0,\ldots,z_n) \in Z^{n+1}$, and obtain a diagram as above where the simplicial space $\cX$ is replaced by the scheme $X$. In particular, for each such $z = (z_0,\ldots,z_n)$ we obtain a morphism of simplicial spaces over $X$:
$$
\xymatrix{
\frX_{H,\Lambda,z} \ar[r] \ar[dr]^{f_z} & \frX_{\Lambda,z} \ar[d]^{g_z}\\
  & X .}
$$
An argument via proper base change and cohomological descent now implies that one has an induced morphism of constructible sheaves $R^{p}g_{z,*}\bbQ \rightarrow R^pf_{z,*}\bbQ$ on $X$ which at the level of stalks at a point $x \in X$ is precisely the morphism of cohomology groups \eqref{equation:mapincohomology} appearing above. Moreover, $R^{p}g_{z,*}\bbQ$ is a constant local system.
\item The sheaves $R^pg_{z,*}\bbQ$ (respectively $R^pf_{z,*}\bbQ$) are abutments of a standard spectral sequence associated to the simplicial space $\frX_{\Lambda,z}$ (respectively $\frX_{H,\Lambda,z}$) whose $\rE_1^{j,i}$-terms are given by $\R^ig_{j,*}\bbQ$ (respectively $\R^if_{j,*}\bbQ$) where 
$g_j: \frX_{\Lambda,z,j} \rightarrow X$ (respectively $f_j: \frX_{H,\Lambda,z,j} \rightarrow X$)   is the induced map on $j$-simplices. Moreover, we have a morphism of spectral sequences from the one for $g$ to that for $f$. By construction, Step~\ref{itemlabel:step1} above allows us to control the lengths of partitions on which the kernels of the resulting morphisms of spectral sequences on the $\rE_1^{j,i}$-terms are constant sheaves along the stratification $\cS$. A delicate analysis of the resulting kernels on the ensuing pages of the spectral sequence then allows us to conclude the analogous result for the abutments of the spectral sequence (see Theorem \ref{thm:spectralseqinput}), and hence concludes the proof of 
Theorem~\ref{thm:main} (in all three versions).
\end{enumerate}

\subsection{Contents}
We briefly describe the contents of the sections below. 

In Section~\ref{sec:applications}, we describe a few applications of our main results. We first show our notion of higher degree VC-density can be extended to higher orders (or arities), parallel to similar notions in the classical case. We then describe two combinatorial applications of the notion of higher degree VC-density. More precisely, we show that under certain situations bounds on the higher degree VC-density implies the  existence of higher degree $\eps$-nets and higher degree fractional Helly number. These are higher degree analogs of the corresponding results in the classical case.

In Section~\ref{sec:prelim}, we recall some basic background from the theory of constructible sheaves in the various setting discussed above: schemes over the complex numbers, schemes over arbitrary algebraically closed field and the \'{e}tale topology, and the o-minimal setting. We recall the aforementioned spectral sequences, and prove the key result on constancy of kernels of morphism of such spectral sequences (Theorem~\ref{thm:spectralseqinput}). We also recall some basic results on cohomological descent adapted to our setting of correspondences.

In Section~\ref{sec:stratification}, we 
prove upper bounds on the number of connected components of realizable $0/1$-patterns of sets of polynomials over arbitrary algebraically closed fields.
We apply these to obtain our desired bounds on lengths of certain stratifications (both over the complex numbers and in the \'{e}tale setting). We also discuss analogous results in the o-minimal setting. 

In Section~\ref{sec:proof:main:complex}, we apply the results of the previous sections in order to prove 
the three versions of Theorem~\ref{thm:main}.
In Section~\ref{sec:tightness} we provide some examples showing the tightness of these bounds.

In Section~\ref{sec:proof:main:general}, we prove 
Theorem~\ref{thm:main:general} 
generalizing 
Theorem~\ref{thm:main}
to the higher order VC-density setting. 

In Section~\ref{sec:proof:applications}, we prove Theorems \ref{thm:epsilon-nets-degree-p} and  \ref{thm:fractional-helly-degree-p}
on the existence of higher degree $\eps$-nets and higher degree fractional Helly theorem respectively.

\section{Applications}
\label{sec:applications}
In this section we discuss applications of the main theorems proved in the paper. Our first application (see Section \ref{subsec:higher-order} below)
is an extension of the notion of higher
degree VC-density introduced in the last section to higher orders
\footnote{
What we call `order' is also referred to as `arity' in the model theory literature.
}
which parallels the generalization of the $\mathrm{NIP}$ property to 
$\mathrm{NIP}_q, q \geq 1$ due to Shelah \cite{Shelah2014}. We prove a generalization of Theorem~\ref{thm:informal} 
(Theorem~\ref{thm:informal:general}) bounding the higher order VC-densities in all the theories considered in this paper.

In Section~\ref{subsec:combinatorial}, we describe some applications to show that knowing a bound on the higher degree VC-densities for certain classes of families of subspaces lead to interesting conclusions -- namely, existence of higher degree $\eps$-nets (that we define) and 
also higher degree fractional Helly numbers (Theorems~\ref{thm:epsilon-nets-degree-p} and \ref{thm:fractional-helly-degree-p}).

\subsection{Higher Order Independence Property}
\label{subsec:higher-order}
While we believe that we are the first to introduce the notion of higher degree VC-density, there is another generalization of the notion
of VC-dimension theory (to higher orders) originating in the work of Shelah \cite{Shelah2014},  who generalized the notion of the independence-property of formulas to higher orders.

A formula
$\phi(x;y^{(0)},\ldots,y^{(q-1)})$ is said to be $q$-independent for some theory $\bT$, if there exists a model $\bbM$ of $\bT$, for every 
$n > 0$, there exists  $Z^{(i)} \subset_n  \bbM^{|y^{(i)}|}, 0 \leq i \leq q-1$,
such that for every subset $S \subset Z^{(0)} \times \cdots \times Z^{(q-1)}$, there exists $x_S \in \bbM^{|x|}$, such that for every 
$(y_0,\ldots,y_{q-1}) \in  Z^{(0)} \times \cdots \times Z^{(q-1)}$ 
$\bbM \models \phi(x_S,y_0,\ldots,y_{q-1})$ if and only 
$(y_0,\ldots,y_{q-1}) \in S$.
The formula $\phi$ is said to be $q$-dependent if it is not $q$-independent and a theory $\bT$ has the property  $\mathrm{NIP}_q$ if every formula is $q$-dependent. It is obvious that the property $\mathrm{NIP}_q$ implies $\mathrm{NIP}_{q+1}$, and the 
property $\mathrm{NIP}_1$ is the same as $\mathrm{NIP}$ defined earlier.
The higher order $\mathrm{NIP}$-property (i.e. $\mathrm{NIP}_q$ for $q>1$) has been studied extensively in recent times 
\cites{Shelah2014,Chernikov-et-al-2019,Beyarslan,Hempel}).
The $\mathrm{NIP}_q$ property motivates the following generalization of the notion of higher degree VC-density to higher order dependence.

Suppose that
$Y^{(0)},\ldots,Y^{(q-1)}$ are sets
and 
$
\mathcal{X}
$
a 
family  of 
subsets of $Y^{(0)} \times \cdots \times Y^{(q-1)}$.

For any tuple of subsets $\bar{Y}' = (Y^{(0)'},\ldots,Y^{(q-1)'})$ 
where $Y^{(i)'} \subset Y^{(i)}$,
we set
\[
\bar{S}(\bar{Y}';\mathcal{X}) := \{\bar{Y}' \cap \bX \mid \bX \in \mathcal{X}\}.
\]

We will also use the convenient notation  
\[
\bar{Y}' = (Y^{(0)'},\ldots,Y^{(q-1)'}) \subset_n \bar{Y}
\] 
to mean that for each $i$
\[
Y^{(i)'} \subset_n Y^{(i)} .
\]

Denote
\begin{equation}
\label{eqn:nu:genera;}
\nu_{\mathcal{X},q}(n) =  \max_{\bar{Y}' \subset_n \bar{Y}} \card(\bar{S}(\bar{Y}';\mathcal{X})),
\end{equation}
and finally define

\begin{equation}
    \label{eqn:def:vd-density:general}
    \vcd_{\mathcal{X},q} = \limsup_n \frac{\log(\nu_{\mathcal{X},q}(n))}{ \log(n)}.
\end{equation}

We refer to $\vcd_{\mathcal{X},q}$ as the  \emph{order $q$ VC-density of $\mathcal{X}$}. 
Our higher degree notion of VC-density generalizes to the setting of higher order 
VC-density as well.

Fix $q \geq 1$, and
suppose that $Y^{(0)},\ldots,Y^{(q-1)}$ are topological spaces, and
$\mathcal{X}$  a set of \emph{closed} subspaces of $Y^{(0)} \times \cdots \times Y^{(q-1)}$, and for each $i, 0 \leq i \leq q-1$,
$\mathcal{Z}^{(i)}$ a set of closed subspaces of $Y^{(i)}$. 
We denote 
$\bar{\mathcal{Z}} = (\mathcal{Z}^{(0)},\ldots,\mathcal{Z}^{(q-1)})$.
Let for each $i,0 \leq i \leq q-1$
$\mathcal{Z}_0^{(i)} \subset_n \mathcal{Z}^{(i)}$,
and
let 
$\bigcup \bar{\mathcal{Z}}_0$ denote  
$\prod_{0 \leq i \leq q-1}
\bigcup_{\bZ^{(i)} \in \mathcal{Z}_0^{(i)}} \bZ^{(i)}$.

For each $p \geq 0$, define
\[
\bbS^p(\bar{\mathcal{Z}}_0;\mathcal{X}) = \{\ker(\HH^p(\bigcup \bar{\mathcal{Z}}_0,\bbQ) \rightarrow \HH^p(\bigcup \bar{\mathcal{Z}}_0 \cap \bX, \bbQ)) \mid \bX \in \mathcal{X}\}.
\]

We define:
\begin{equation}
\label{eqn:def:vcd:pq}
\vcd^{p,q}_{\mathcal{X}} := \sup_{\bar{\mathcal{Z}}} \limsup_{n}  \frac{ \log(\nu_{\mathcal{X},\bar{\mathcal{Z}}}^{p,q}(n))}{\log(n)},
\end{equation}
where
\begin{equation}
\label{eqn:chi-C:pq}
\nu^{p,q}_{\mathcal{X},\bar{\mathcal{Z}}} (n) =  
\sup_{\mathcal{Z}_0^{(i)} \subset_n \mathcal{Z}^{(i)}, 0 \leq i \leq q-1} \card(\bbS^p(\bar{\mathcal{Z}}_0;\mathcal{X})),
\end{equation}
and the supremum in \eqref{eqn:def:vcd:pq} is taken over all 
tuples $(\mathcal{Z}^{(0)},\ldots,\mathcal{Z}^{(q-1)})$ with
$\mathcal{Z}^{(i)}$ a set of closed subspaces of $Y^{(i)}$ for $0 \leq i \leq q-1$. In the definable setting, 
each $\mathcal{Z}^{(i)}$ is a 
closed
definable family of 
subspaces of
$Y^{(i)}$. The outer supremum is taken over
tuples of such families.

We have following generalization of Theorem~\ref{thm:informal}. 

Let $\bbM$ be an o-minimal expansion of $\mathbb{R}$ or a model
of one of the following theories: 
$\mathrm{RCF}$,
$\mathrm{ACF}(0)$ or $\mathrm{ACF}(p)$. 

Suppose that
$X,Y^{(0)},\ldots,Y^{(q-1)}$,
are 
definable subsets in 
$\bbM$
and 
$H \subset X \times Y^{(0)} \times \cdots \times  Y^{(q-1)}$
a closed definable subset. 
Then, 
\[
\mathcal{X} := \{ \leftindex_{x}H \mid x \in X\}
\]
is a 
closed definable family  of 
subsets of $Y^{(0)} \times \cdots \times Y^{(q-1)}$,
where 
\[
\leftindex_{x}H = \pi_{Y^{(0)} \times \cdots \times Y^{(q-1)}} (\pi_X^{-1}(x) \cap H), 
\]
and 
$\pi_X: X \times Y^{(0)} \times \cdots Y^{(q-1)} \rightarrow X$, $\pi_{Y^{(0)} \times \cdots \times Y^{(q-1)}}: X \times Y^{(0)} \times \cdots \times Y^{(q-1)} \rightarrow Y^{(0)} \times \cdots \times Y^{(q-1)}$
are the projection maps.

\begin{theorem}
\label{thm:informal:general}
Suppose that $Y^{(0)},\ldots,Y^{(q-1)}$ are proper. 
For  each $p \geq 0$ and  $q \geq 1$, 
\[
\vcd^{p,q}_{\mathcal{X}} \leq (p+q) \cdot \dim X.
\]
\end{theorem}

Notice that \[
\vcd^{p,1}_{\mathcal{X}} = \vcd^{p}_{\mathcal{X}}.
\]

Also note that by taking $Y = Y^{(0)} \times \cdots \times Y^{(q-1)}$, 
considering $\bar{\mathcal{Z}}$ as a subset of $2^{Y}$
it follows directly from \eqref{eqn:chi-C:p} and \eqref{eqn:chi-C:pq} that,
writing
$\mathcal{W}  = \{Z_0 \times \cdots \times Z_{q-1} \;\mid\; Z_i \in \mathcal{Z}^{(i)}\}$
\[
\nu^{p,q}_{\mathcal{X},\bar{\mathcal{Z}}} (n)
\leq 
\max_{1 \leq r \leq n^q}
\nu^{p}_{\mathcal{X},\mathcal{W}} (r)
\]
(since the maximum is being taken over a smaller set of choices in 
\eqref{eqn:chi-C:pq}).

It now follows directly from Theorem~\ref{thm:informal}, that
\[
\vcd^{p,q}_{
\mathcal{X}} \leq 
q(p+1) \cdot \dim X.
\]

Thus, if $p=0$ or $q=1$, Theorem~\ref{thm:informal:general} follows immediately from Theorem~\ref{thm:informal}. In every other case,
the bound in Theorem~\ref{thm:informal:general} is stronger than the one 
obtained by applying Theorem~\ref{thm:informal}.

\begin{remark}
Our generalized notion of VC-density is graded by bi-degree $(p,q)$ where $p \geq 0$ and $q \geq 1$ are integers. We refer to the index $p$ as the \emph{degree} and the index $q$ as the \emph{order} of the VC-density. 
For any definable family $\mathcal{X}$ to which Theorem~\ref{thm:informal} is applicable 
\[
\vcd_{\mathcal{X}} \leq \vcd_{\mathcal{X}}^{0,1}. 
\]
So an upper bound on the degree $0$ and order $1$ VC-density of a family $\mathcal{X}$ is also an upper bound on the classical VC-density of $\mathcal{X}$.

More generally, for every $q \geq 1$,
\[
\vcd_{\mathcal{X},q} \leq \vcd_{\mathcal{X}}^{0,q}. 
\]
Thus, an upper bound 
on the degree $0$ and order $q$  VC-density of $\mathcal{X}$,
is also an upper bound on the order $q$ VC-density of $\mathcal{X}$. 
\end{remark}

\begin{remark}
Note also that the generalized VC-density
$\vcd^{p,q}_{\mathcal{X}}$
measures the 
`complexity' of the definable family $\mathcal{X}$  against
collections of $n^q$ subsets of $Y^{(0)} \times \cdots \times Y^{(q-1)}$ of the special (product) form $\mathcal{Z}^{(0)}_0 \times \cdots \times \mathcal{Z}^{(q-1)}_0$, where each $\mathcal{Z}^{(i)}_0 \subset_n \mathcal{Z}^{(i)}$
(rather than against arbitrary subsets of size $n^q$ of
$\mathcal{Z}^{(0)} \times \cdots \times \mathcal{Z}^{(q-1)}$).
Since $p+q = q(p+1)$ whenever $p=0$ or $q =1$, the difference 
between these two classes of `test' families of finite subsets,
(i.e. finite sets of cardinality $n^q$ of the special form 
$\mathcal{Z}^{(0)}_0 \times \cdots \times \mathcal{Z}^{(q-1)}_0$, as opposed to general subsets of 
$\mathcal{Z}^{(0)} \times \cdots \times \mathcal{Z}^{(q-1)}$
having cardinality $n^q$)
is reflected in our bound (Theorem~\ref{thm:informal:general})
only for $p > 0$ and $q > 1$. 

\hide{
It follows that for $p >0 $ and $q > 1$, the higher order and degree VC densities 
$\vcd_{\mathcal{X}}^{p,q}$
are sensitive to the product structure of finite sets
in a way that the classical VC-density,
or even the higher order versions of Sauer-Shelah, namely $\vcd_{\mathcal{X},q}, q>1$,
are not.
}
\end{remark}

\subsubsection{Upper bound on $\nu^{p,q}_{\mathcal{X},\bar{\mathcal{Z}}}$}
Theorem~\ref{thm:informal:general} is a consequence on a quantitative upper bound on the function
$\nu^{p,q}_{\mathcal{X},\bar{\mathcal{Z}}}$
(for appropriate family $\mathcal{X}$ and tuples of families
$\bar{\mathcal{Z}}$), in the same way as
Theorem~\ref{thm:informal} is a consequence of the quantitative
bound in
Theorem~\ref{thm:main}. 

Let $q \geq 1$ be fixed, and
$[H;X,Y_0 \times \cdots \times Y_{q-1} ]$ and $[\Lambda_i;Y_i,Z_i],0 \leq i \leq q-1$,
correspondences satisfying the same properties as in Theorem~\ref{thm:main}
(in one of the three categories in Table~\ref{tab:versions}).

In each of the three categories we have the following theorem.

\begin{theorem}
\label{thm:main:general}
Let $p \geq 0$.
There exists a constant $C = C_{H,\Lambda_0,\ldots,\Lambda_{q-1},p} > 0$ (depending only on the correspondences 
$[H;X,Y_0 \times \cdots \times Y_{q-1}]$,
$[\Lambda_i;Y_i,Z_i], 0 \leq i \leq q-1$,  and $p$), such that for all $n \geq 1$,
\[
\nu^{p,q}_{\mathcal{X},\bar{\mathcal{Z}}} (n) \leq C \cdot n^{(p+q)\dim X},
\]
where $\mathcal{X} = \{\leftindex_{x} H \mid x \in X\}$ and 
$\bar{\mathcal{Z}} = (\mathcal{Z}^{(0)},\ldots, \mathcal{Z}^{(q-1)})$ with
$\mathcal{Z}^{(i)} = \{\Lambda_{i,{z_i}} \mid z_i \in Z_i\}$.
\end{theorem}

\begin{remark}
\label{rem:product-family}
Applying Theorem~\ref{thm:main} instead to the product correspondence
\[
\Bigl[\,\prod_{i=0}^{q-1}\Lambda_i\,;\;
       \prod_{i=0}^{q-1}Y_i\,,\;\prod_{i=0}^{q-1}Z_i\Bigr],
\]
and to the $n^q$ fibers
$\Lambda_{0,z_0}\times\cdots\times\Lambda_{q-1,z_{q-1}}$
obtained from $n$-element subfamilies of the
$\mathcal{Z}^{(i)}$, $0\leq i\leq q-1$, gives only the weaker bound
$C'\cdot n^{q(p+1)\dim X}$, with a possibly different constant $C'$
(compare the discussion following Theorem~\ref{thm:informal:general}).
\end{remark}

\begin{remark}
\label{rem:two-proofs}
Setting $q=1$ in Theorem~\ref{thm:main:general} recovers
Theorem~\ref{thm:main}.
The proof of Theorem~\ref{thm:main:general} in
Section~\ref{sec:proof:main:general} is independent of the proof of
Theorem~\ref{thm:main}, and for $q=1$ it gives a second proof of that
theorem.
We retain the proof of Theorem~\ref{thm:main} in
Section~\ref{sec:proof:main:complex} because it realises the kernels
as the stalks of a constructible sheaf on $X$, via the spectral
sequence of a simplicial resolution as outlined in Section~\ref{subsec:method},
and highlights the key ideas behind the proof.
\end{remark}

\subsection{Combinatorial applications}
\label{subsec:combinatorial}
The tight upper bound that we prove on the higher degree VC-density 
should lead to higher degree versions of combinatorial results 
in which the classical degree $0$ VC-density plays a role.
Upper bounds on the (degree $0$) VC-density of certain families of subsets of 
a fixed ambient space play an important role in many applications in 
discrete geometry (see for example the book \cite{Matousek-book}).
We consider in this paper two such applications and extend these 
to the higher degree situation.

We first begin with some notation that should be seen as a higher degree
analog of `$\in$'.
\begin{notation}
\label{not:pitchfork}
 Let $Y$ be a topological space and $\mathcal{X},\mathcal{Z}$ sets of subspaces of $Y$, and $p \geq 0$. For $\bX \in \mathcal{X}, \bZ \in \mathcal{Z}$, we denote 
 \[
 \bZ \in_p \bX
 \]
 if
 the restriction homomorphism $\HH^p(\bZ,\bbQ) \rightarrow \HH^p(\bZ \cap \bX,\bbQ)$ is non-zero.
\end{notation}

\begin{remark}[Geometric interpretation of $\in_p$]
    In some cases, Notation~\ref{not:pitchfork} can be seen
    as belonging. 
    Indeed, if $\bZ$ is a point, then 
    \[
     \bZ \in_0 \bX \Longleftrightarrow \bZ \in \bX.
    \]
    More generally,  if $\bZ$ is an irreducible closed subscheme of an affine scheme $S$ with $\HH^{\dim Z}(\bZ,\bbQ) \neq 0$, and $\bX$ a closed subscheme of $S$, 
    then
    \[
    \bZ \in_{\dim \bZ} \bX \Longleftrightarrow \bZ \subset \bX.
    \]
\end{remark}

\subsubsection{Higher degree VC-density bounds, $\eps$-nets 
and the fractional Helly number
}
One key property of set systems that is ensured by the finiteness
of the VC-density is the existence of $\eps$-nets of constant size.
This fact is of great importance in discrete and computational geometry
(see for example \cite{Matousek-book}).

We recall here the definition of $\eps$-nets for set systems
(see for example \cite{Matousek-book}*{Definition 10.2.1})

\begin{definition} 
\label{def:eps-net}
Let $Y$ be a finite set and $\mathcal{X} \subset 2^Y$. 
For $0 < \eps <1$, a subset $S \subset Y$ is called an $\eps$-net
for $(Y,\mathcal{X})$, if it satisfies the property that for all
$\bX \in \mathcal{X}$ with $\card(\bX) \geq \eps \cdot  \card(Y)$, 
$S \cap \bX \neq \emptyset$.
\end{definition}

Recall the growth function introduced earlier: for a set $Y$
and a family $\mathcal X\subseteq 2^Y$,
\[
\nu_{\mathcal X}(m)
=\max_{Y'\subset_mY}\card\bigl(S(Y';\mathcal X)\bigr),
\qquad
S(Y';\mathcal X)=\{Y'\cap X:X\in\mathcal X\}.
\]
Here $Y'\subset_mY$ means that $Y'$ has exactly $m$ distinct
elements; the maximum is taken only for sample sizes for which
such a subset exists.  The $\eps$-net theorem of Haussler and
Welzl~\cite{Haussler-Welzl}, and the sharper estimates of
Koml\'os, Pach, and Woeginger~\cite{Komlos-Pach}, relate the size
of an $\eps$-net to the growth of this function.
The argument of Koml\'os--Pach--Woeginger yields the following
asymptotic consequence in terms of the growth function: if
$\nu_{\mathcal X}(m)\leq A m^s$ for all admissible $m\geq1$,
where $A\geq1$ and $s>0$ are fixed, then every nonempty finite
$Y_0\subseteq Y$, equipped with the uniform probability measure,
admits an $\eps$-net of size at most
\[
\frac{s+o_{A,s}(1)}{\eps}\ln\frac1\eps
\qquad(\eps\downarrow0),
\]
uniformly in $Y_0$.  Consequently:
\begin{proposition}
\label{prop:Pach}
 If $s<d$, there is an
$\eps_0=\eps_0(A,s,d)>0$ such that 
there exists an $\eps$-net of size
at most
$d\eps^{-1}\ln(1/\eps)$ for every $0<\eps<\eps_0$.   
\end{proposition}

In particular, when $Y$ is infinite and $\vcd_{\mathcal X}<d$,
one may choose $\vcd_{\mathcal X}<s<d$ and use the definition
of VC-density to obtain such a polynomial bound.  Thus the
threshold $\eps_0$ may depend on the fixed family $\mathcal X$
and on $d$, but is independent of the finite subset $Y_0$.

We now extend the notion of $\eps$-nets to higher degrees.
We call a topological space $Y$ and two sets of subspaces
$\mathcal{X}$ and $\mathcal{Z}$ a triple, denoted by $(Y,\mathcal{X},\mathcal{Z})$.

\begin{definition}[$\eps$-nets in degree $p$]
\label{def:epsilon-nets-degree-p}
    Let $(Y,\mathcal{X},\mathcal{Z})$ be a triple, and
    $p \geq 0, 0 < \eps < 1$. Given  a finite subset $\mathcal{Z}_0 \subset \mathcal{Z}$, we say 
    $\mathcal{S} \subset \mathcal{Z}_0$, 
    is a \emph{degree $p$ $\eps$-net for $\mathcal{Z}_0$},
    if for all $\bX \in \mathcal{X}$, such that 
    \[
    \card(\{\bZ \in \mathcal{Z}_0 \mid \bZ \in_p \bX\}) \geq \eps \cdot \card(\mathcal{Z}_0),
    \]
    there exists $\bZ \in \mathcal{S}$ such that $\bZ \in_p \bX$.
\end{definition}

\begin{remark}[$\eps$-nets for finite sets]
\label{rem:epsilon-nets-finite}
    The notion of $\eps$-nets for finite subsets of $Y$ with respect to the family $\mathcal{X}$ can be recovered by taking $p=0$, and 
    $\mathcal{Z} = \{\{y\} \mid y \in Y\}$. 
\end{remark}

\begin{theorem}[Asymptotic existence of $\eps$-nets in degree $p$]
\label{thm:epsilon-nets-degree-p}
Let $Y=\bbR^N$, and let $p\geq0$ and $d\geq1$ be integers.
Let $\mathcal X$ and $\mathcal Z$ be families of closed
semi-algebraic subsets of $Y$, with $\dim Z\leq p$ for every
$Z\in\mathcal Z$.  Suppose that
\[
\rho:=\vcd^p_{\mathcal X,\mathcal Z}<d.
\]
There exists $\eps_0=\eps_0(\mathcal X,\mathcal Z,p,d)>0$
such that, for every nonempty finite
$\mathcal Z_0\subseteq\mathcal Z$ and every $0<\eps<\eps_0$,
there is a degree-$p$ $\eps$-net
$\mathcal S\subseteq\mathcal Z_0$ satisfying
\[
|\mathcal S|\leq\frac d\eps\ln\frac1\eps.
\]
Thus, whenever $X\in\mathcal X$ satisfies
\[
\card(\{Z\in\mathcal Z_0:Z\in_pX\})\geq\eps|\mathcal Z_0|,
\]
some $Z\in\mathcal S$ satisfies $Z\in_pX$.
\end{theorem}

Here $Z\in_pX$ means that the restriction
$\HH^p(Z,\Q)\longrightarrow\HH^p(Z\cap X,\Q)$ is nonzero.

\begin{remark}
For example, let $\mathcal Z$ be a family of irreducible real
algebraic curves in $\bbR^3$, each homeomorphic to $\bbS^1$,
and let $\mathcal X$ be a family of real algebraic surfaces.
If a curve $Z$ is not contained in a surface $X$, their
intersection is finite, so $\HH^1(Z\cap X,\bbQ)=0$.
If $Z\subset X$, the restriction from $Z$ to $Z\cap X=Z$
is the identity.  Thus $Z\in_1X$ means precisely $Z\subset X$.
Under the density hypothesis of the theorem, every sampled
family of curves therefore has a small subfamily meeting every
surface that contains an $\eps$-fraction of the sample.
This incidence does not detect whether a contained circle is
contractible or null-homologous inside the surface.
\end{remark}

\begin{remark}
The families in Theorem~\ref{thm:epsilon-nets-degree-p} need not
be uniformly definable.  This is why its hypothesis uses
$\vcd^p_{\mathcal X,\mathcal Z}$, which is defined for a fixed
pair of families, rather than $\vcd^p_{\mathcal X}$, whose
definition in this paper is restricted to definable families.
When both families are uniformly definable, the latter density
bound does imply the required bound for the fixed pair.
\end{remark}

Another application of the classical notion of VC-density is related to the fractional Helly number. We now recall this application, and our generalization to higher degrees.

\begin{definition}[Fractional Helly number]
Let $Y$ be a set, $\mathcal Z\subset 2^Y$, and $k\geq1$ an integer.
The pair $(Y,\mathcal Z)$ has fractional Helly number at most $k$
if for every $0<\alpha\leq1$ there is $\beta>0$ such that the
following holds for every indexed family $Z_1,\ldots,Z_n\in\mathcal Z$
with $n\geq k$; repetitions are allowed.
If at least $\alpha\binom nk$ of its $k$-element sets of indices
$I$ satisfy $\bigcap_{i\in I}Z_i\neq\emptyset$, then some
$y\in Y$ belongs to at least $\beta n$ of the indexed sets.
\end{definition}

The following theorem links dual VC-density with fractional Helly number.

For a family $\mathcal Z\subset 2^Y$, define its dual shatter
function by
\[
\pi^*_{\mathcal Z}(m)=
\max\left\{1,\sup_{Z_1,\ldots,Z_m\in\mathcal Z}
 \card\bigl(\{\{i:y\in Z_i\}:y\in Y\}\bigr)\right\}.
\]
Repeated sets are allowed in this definition and do not increase
the number of patterns.  Write
\[
\vcd^*_{\mathcal Z}
 =\limsup_{m\to\infty}\frac{\log\pi^*_{\mathcal Z}(m)}{\log m}
\]
for the dual VC-density, or VC-codensity, of $\mathcal Z$.

\begin{theorem}\cite[Theorem 2]{Matousek2004}
\label{thm:fractional-Helly}
Let $Y$ be a set, $\mathcal Z\subset 2^Y$, and let $d\geq1$
be an integer.  If
\[
\pi^*_{\mathcal Z}(m)=o(m^d),
\]
then $(Y,\mathcal Z)$ has fractional Helly number at most $d$.
In particular, it suffices that $\vcd^*_{\mathcal Z}<d$.
(The hypothesis is on the VC-codensity of $\mathcal Z$.)
\end{theorem}

\begin{proof}
The little-$o$ assertion is the cited theorem of Matou\v{s}ek.
If $\vcd^*_{\mathcal Z}<d$, choose a real number $s$ strictly
between these two numbers.  For all sufficiently large $m$,
$\pi^*_{\mathcal Z}(m)\leq m^s$, and hence
$\pi^*_{\mathcal Z}(m)/m^d\to0$.
\end{proof}

We now formulate below a higher degree analog of the fractional Helly number.

\begin{definition}[Fractional Helly number in degree $p$]
\label{def:fractional-helly-degree-p}
Let $(Y,\mathcal X,\mathcal Z)$ be a triple and $k\geq1$ an integer.
It has degree-$p$ fractional Helly number at most $k$ if for every
$0<\alpha\leq1$ there is $\beta>0$ such that, for every indexed
family $Z_1,\ldots,Z_n\in\mathcal Z$ with $n\geq k$, the following
implication holds.  If at least $\alpha\binom nk$ sets of indices
$I$ of size $k$ admit an $X_I\in\mathcal X$ with
$Z_i\in_pX_I$ for every $i\in I$, then there is one
$X\in\mathcal X$ with
\[
\card\{i\in\{1,\ldots,n\}:Z_i\in_pX\}\geq\beta n.
\]
Repeated test sets are allowed.
\end{definition}

\begin{remark}
\label{rem:fractional-helly-degree-0}
    The ordinary notion of fractional Helly number of the family $\mathcal{Z}_0$ is recovered by taking $p=0$ and the family
    $\mathcal{X} = \{\{y\} \mid y \in Y\}$ 
    as in Remark~\ref{rem:epsilon-nets-finite}. Observe that in this case
    `there exists $\bX_I \in \mathcal{X}$ satisfying
$\bZ_i \in_p \bX_I$ for every $i \in I$' translates to
`$\bigcap_{i \in I} \bZ_i \neq \emptyset$'.
\end{remark}

\begin{theorem}[Fractional Helly's theorem in degree $p$]
\label{thm:fractional-helly-degree-p}
Let $p,d,(Y,\mathcal X,\mathcal Z)$ satisfy the hypotheses of
Theorem~\ref{thm:epsilon-nets-degree-p}.
Then, $(Y,\mathcal{X},\mathcal{Z})$ 
has degree $p$ fractional Helly number bounded by $d$.
\end{theorem}


\section{Preliminary Background and results}
\label{sec:prelim}
In this section we recall some facts and results that we will need in the
proof of Theorem~\ref{thm:main}. Since these facts are more easily accessible in the complex analytic setting (i.e. in the second category 
of Table~\ref{tab:versions}) we start first with this category, and 
explain later the corresponding results in the other two categories.

\subsection{Constructible sheaves in complex analytic topology}
In this subsection, we work with schemes $X$ of finite type over $\mathbb{C}$. Given such a scheme, one has an associated complex analytic space $X^{an}$ with its underlying complex topology. In the following, we fix a commutative noetherian ring $R$ of finite Krull dimension, and consider sheaves of $R$-modules on $X$ in the complex analytic topology (i.e. it is a sheaf on $X^{an}$, but by abuse of notation we shall refer to these as sheaves on $X$). We denote by $\Sh(X,R)$ the abelian category of sheaves of $R$-modules, and by $\rD^{b}(X,R)$ the corresponding derived categories. Given a morphism $f: X \rightarrow Y$ (of finite type), we denote by $f^*: \Sh(Y,R) \rightarrow \Sh(X,R)$ and $f_*: \Sh(X,R) \rightarrow \Sh(Y,R)$ the usual pull-back and push-forward functors. We remind the reader that $f^{*}$ is an exact functor, while $f_*$ is a left-exact functor. We denote by $Rf^*$ and $Rf_*$ the resulting derived functors (on the corresponding derived categories), and $R^if_*$ the $i$-th cohomology of $Rf_*$. Note that if $f: X \rightarrow \Spec(\bbC)$ is the structure morphism, then $R^if_*(\cF) = \rH^i(X,\cF)$ for a sheaf $\cF$ on $X$. Here the right hand side is the usual sheaf cohomology groups (on $X^{an})$.

\subsubsection{Stratifications}
A {\it stratification} of $X$ is a finite collection of locally closed (in the Zariski topology) subsets  $S_i \subset X$ ($i \in  I$) (considered as a subscheme with its canonical induced reduced structure) such that $X = \coprod S_i$. We refer to $|I|$ as the {\it length} of the stratification. We refer to the $S_i$ as strata (or stratum) and use the notation $\cS$ to denote the collection of strata $S_i$.
Suppose $X = X_0 \supsetneq X_1 \supsetneq X_2 \supsetneq \cdots \supsetneq X_n \neq \emptyset$ is a filtration of $X$ by (Zariski) closed subsets. Then we may associate a canonical stratification of length $n +1$. We set $S_i := X_{i} \setminus X_{i+1}$ for $0 \leq i < n$ and $S_n = X_n$. Note that each $S_i$ is locally closed. A {\it refinement} of a stratification $\cS$ is a stratification $\cS'$ such that each stratum $S_i \in \cS$ is a union of strata $S_i' \in \cS'$. 

\subsubsection{Locally constant sheaves}
Below, we assume $X$ is connected. Given an $R$-module $M$, we denote by $\underbar{M}$ the constant sheaf on $X$. In particular, $\underbar{M}$ is the sheaf associated to the constant pre-sheaf whose values on any open is the $R$-module $M$ (with identity maps as restriction maps). We note that the sections of $\underbar{M}$ on an open $U$ are given by locally constant functions $U \rightarrow M$ where $M$ is given the discrete topology. Note that if $U$ is connected, then such a function must be constant, and therefore the sections of $\underbar{M}$ on such an open are given by $M$. A {\it locally constant sheaf} $\cF$ on $X$ is a sheaf $\cF$ such that for every $x \in X$ there is an open neighborhood $x \in U$ (in the complex analytic topology) so that the restriction of $\cF$ to $U$ is a constant sheaf (by restriction we mean the pull-back of $\cF$ along the inclusion morphism). The category of locally constant sheaves is a full abelian sub-category of $\Sh(X,R)$. 

\begin{remark}
\label{rem:localsystemanalytic}
If we fix a base point, then we may identify the category of local systems of $R$-modules on $X$ with the category of representations $\phi: \pi_1(X,x) \rightarrow \Aut_R(M)$ of the fundamental group of $X^{an}$ in $R$-modules $M$ (\cite{Del-regsing}, 1.3).\footnote{Note that in loc. cit. this is proved for topological spaces which are locally path connected and locally simply path connected. These conditions are automatic for the complex analytic spaces we consider.} The constant sheaves identify with the trivial representations. The functor sends a local system $\cL$ to its stalk at $x$. One can deduce the following three properties from this equivalence of categories:
\begin{enumerate}
\item The category of local systems of $R$-modules is an abelian category. This can also be deduced directly from the definitions.
\item If $f: \underbar{M} \rightarrow \cL$ is a morphism from a constant local system to a local system, then the image is a constant local system. The point is that the monodromy action on the image is trivial.
\item As a consequence of the last part, the kernel of $f$ is a constant local system. To see this note that the kernel of a morphism of {\it constant} local systems $\underbar{M} \rightarrow \underbar{N}$ is the constant local system associated with the kernel of the morphism $M \rightarrow N$ induced on global sections (we assume $X$ is connected).
    
\end{enumerate}
\end{remark}

Let $\pi: \cF \rightarrow \cG$ be a morphism of local systems of $R$-modules on $X$. Suppose that $\cF$ is a constant local system associated to an $R$-module $M$. For each $x \in X$, let $\pi_x: M = \Gamma(X,\cF) \rightarrow \Gamma(X, \cG) \rightarrow \cG_x$ denote the resulting morphisms from the global sections of $\cF$ to the stalk of $\cG$ at $x \in X$. Here the second arrow is the usual restriction map sending a global section to its germ at $x \in X$. Let $M(\pi,x) \subset M$ denote the kernel of $\pi_x$.
The following proposition will be a key tool in the following.

\begin{proposition}\label{prop:kernelisindep}
With notation as above, $M(\pi,x)$ is independent of $x$.
\end{proposition}
\begin{proof}
Let $\cN$ be the kernel of $\pi$. By the previous remark, this is a constant local system. On the other hand, the kernels $M(\pi,x)$ are by definition the kernel of the stalks $M = \underbar{M}_x \rightarrow \cG_x$ (recall that the stalk $\underbar{M}_x =M$ and the restriction map $M = \underbar{M} \rightarrow M = \underbar{M}_x$ is the identity map). It follows that $\cN_x = M(\pi,x)$, and $\cN_x =\Gamma(X, \cN) \subset M$.
\end{proof}

We record the following lemma for future use.

\begin{lemma}\label{lem:projectionisconstant}
Let $A$ be proper, $q: X \times A \rightarrow X$ and $p: X \times A \rightarrow A$ be the natural projection maps, and $\cG$ a sheaf on $A$. Then $Rq_*(p^*\cG) \isom \underline{R\Gamma(A,\cG)}$. In particular,  $R^iq_*(p^*\cG)$ is the constant sheaf given by $\underline{\rH^i(A,\cG)}.$
\end{lemma}
\begin{proof}
 We have 
$$Rq_*(p^*\cG) \isom \underline{R\Gamma(A,\cG)},$$
as a consequence of proper base change applied to the cartesian diagram:
$$
\xymatrix{
X \times A \ar[r]^p \ar[d]^q & A \ar[d] \\
X  \ar[r]   &  \Spec(\bbC).
}
$$
\end{proof}

\subsubsection{Constructibility}\label{subsubsec:constructible}
A sheaf $\cF \in \Sh(X,R)$ is $\cS$-{\it constructible} for a stratification $\cS$ if the restriction of $\cF$ to each stratum is a locally constant sheaf and all its stalks are finitely generated $R$-modules (see \cite{Dimca-book}*{Chapter 4} or \cite{Deligne-4.5}*{page 43}). 
A sheaf is {\it constructible} if there is a stratification of $X$ for which the sheaf is constructible. An object $\cF \in \rD^b(X)$ is $\cS$-constructible (respectively constructible) if all its homology sheaves are $\cS$-constructible (respectively constructible). We record the following facts for future reference:
\begin{enumerate} 
\item The full subcategory $\Sh_{\cS}(X,R)$ (respectively $\Sh_{c}(X,R)$) of $\cS$-constructible (respectively constructible) sheaves is an abelian sub-category. 
\item Let $\rD^b_{c}(X,R) \subset \rD^b(X,R)$ denote the full sub-category of constructible sheaves. The functors $Rf^*$ and $Rf_*$ preserve the category of constructible sheaves.
\end{enumerate}

\subsection{Simplicial objects of a category $\cA$}
Let $\Delta$ denote the usual simplex category. Recall, the objects of $\Delta$ are given by totally ordered sets $[n] := \{0,\ldots,n\}$ and morphisms are given by order preserving morphisms. Among such morphisms we have the standard face and degeneracy maps:
\begin{enumerate}
\item[1:](Face Maps) These are the inclusions $d_i:[n] \hookrightarrow [n+1]$ which `skip' $i$ (for $0\leq i \leq n+1)$.
\item[2:](Degeneracy Maps) These are the surjections $\delta_i: [n] \rightarrow [n-1]$ which repeat $i$ (for $0 \leq i \leq n-1$).
\end{enumerate}
Let $\Delta[k]$ denote the full subcategory of $\Delta$ consisting of the objects $[0],\ldots,[k]$. We also consider the extended category $\Delta_+$ defined by introducing an initial object denoted by $\emptyset$ (or sometimes $-1$). One also has the analogous extended category $\Delta[k]_+$.

\begin{definition}
A ($k$-truncated) simplicial object of a category $\cC$ is a functor $F: \Delta^{op} \rightarrow \cC$ (respectively $\Delta[k]^{op} \rightarrow \cC$). In other words, it is a $\cC$-valued presheaf on $\Delta$ (respectively $\Delta[k]$). Let $\cC^{\bullet}$ (respectively $\cC^{\leq k}$) denote the category of simplicial (respectively $k$-truncated simplicial) objects (with morphisms given by morphisms of pre-sheaves). An augmented (respectively $k$-truncated) simplicial object is a $\cC$-valued presheaf on $\Delta_+$ (respectively $\Delta[k]_+$).
\end{definition}

\begin{example}
In the following, $\cC$ will be one of the following categories:
\begin{enumerate}
\item The category of schemes (finite type) over an algebraically closed field $k$.
\item The category of topological spaces.
\item The category of definable topological spaces in some 
o-minimal expansion of $\mathbb{R}$
(regarded, for example, as a sub-category of the previous example).
\end{enumerate}
We note that finite limits exist in all of the aforementioned examples, and in particular fiber products exist.
\end{example}

\begin{example}
(i) An object $Z \in \cC$ gives rise to the constant simplicial object, denoted $Z_{\bullet}$, whose $n$-th term is $Z$ and all the face and degeneracy maps are the identity morphism.\\
(ii) If $X \rightarrow S$ is a morphism in $\cC$, and $\cC$ has fiber products, then the 0-th coskeleton of $X \rightarrow S$ is given by the simplicial object whose $k$-th term is $X \times_S \cdots \times_S X$ where the product is taken $k+1$-times. The face maps are given by the natural projection maps. The degeneracy maps are given by various `multi-diagonal' maps. We denote by $\cosk(X/S)$ the resulting simplicial object.\\
\end{example}

\subsection{Constructible sheaves in the \'{e}tale topology and definable spaces}
In this section, we recall some basic facts on the analogs of the constructions of the previous sections in the setting of constructible sheaves in the \'{e}tale topology and in the o-minimal setting. 
The use of the \'{e}tale topology will allow us to extend the VC-density results to the $\mathrm{ACF}$  case (for arbitrary algebraically closed fields including those of characteristic $p > 0$) and the latter to the o-minimal setting.

\subsubsection{The \'{E}tale setting}\label{subsec:constructibleetale}
We fix a base field $k$ of characteristic $p$  (possibly equal to 0) and a prime $\ell \neq p$. Let $\Lambda = \bbZ/\ell\bbZ$ with $\ell \neq p$. Given a scheme $X$, we can consider the corresponding \'{e}tale topology. We denote by $X_{et}$ the resulting \'{e}tale site. In this setting, one can define the analogous notions of local systems, and constructible sheaves in the \'{e}tale topology. We briefly recall the definitions and some properties of the resulting categories which will be used in the following. 

As in the complex analytic case, a $\Lambda$-module $M$ defines a constant sheaf $\underbar{M}$ in the \'{e}tale topology i.e. an object of the category $\Sh(X_{et},\Lambda)$ of sheaves of $\Lambda$-modules in the \'{e}tale topology. A {\it local system} (or locally constant sheaf) $\cL$ of $\Lambda$-modules on $X$ is a sheaf of finite $\Lambda$-modules which is locally constant i.e. there is an open cover $U_i$ in the \'{e}tale topology of $X$ such that $\cL|_{U_i}$ is a constant sheaf. 

\begin{remark}\label{rem:localsystemsetale}
In the aforementioned setting, and for a geometric point $x \in X$, one can define the \'{e}tale fundamental group $\pi_1^{et}(X,x)$. The analog of Remark \ref{rem:localsystemanalytic} continues to hold in this setting with the usual fundamental group replaced by the \'{e}tale fundamental group. In particular:
\begin{enumerate}
 \item The category of local systems of $\Lambda$-modules in the \'{e}tale topology is a full abelian subcategory of the $\Sh(X_{et},\Lambda).$
 \item If $f: \cL \rightarrow \cM$ is a morphism of local systems where $\cL$ is the constant local system, then $ker(f)$ is a constant local system.
\end{enumerate}
\end{remark}

One can define the notion of constructible sheaves of $\Lambda$-modules as in the complex analytic case. In particular, a sheaf $\cF$ of $\Lambda$-modules is constructible with respect to a stratification $\cS$ if the restriction of $\cF$ to each stratum is locally constant. The analogs of the assertions with regards to the derived category of constructible sheaves and $\cS$-constructible sheaves from \ref{subsubsec:constructible}
continue to hold in the \'{e}tale setting.

We note that in the discussion above one may take $\Lambda$ to be more generally any finite ring (prime to the characteristic of the base field). It is also possible to define analogous notions with $\bbZ_{\ell}$ or $\bbQ_{\ell}$-coefficients. However, in this case the definitions of local systems and constructible sheaves are more subtle. Moreover, the analog of Remark \ref{rem:localsystemsetale} is false in general (it is true for geometrically unibranch schemes). One way to remedy the situation is to replace the \'{e}tale topology by the pro-\'{e}tale topology defined by Bhatt and Scholze \cite{BhattScholzeProetale}. With finite coefficients, the resulting categories of local systems and constructible sheaves are the same. We do not recall the details and only note that our main theorems continue to hold for \'{e}tale cohomology with coefficients in $\bbQ_{\ell}$ (or $\bar{\bbQ}_{\ell}$).

\subsubsection{The o-minimal setting}\label{subsubsec:o-minimalsheaves}
Consider an o-minimal expansion of $\mathbb{R}$. The topology on $\mathbb{R}^n, n >0$ is the Euclidean topology generated by open balls.  
One can define the analogous notions of local systems, and constructible sheaves in the above topology. 
A stratification of a definable set $X$ is a finite partition of $X$
into locally closed definable sets.
Local systems and constructible sheaves are defined as in the complex analytic case using the above notion of stratification (see for example \cite{Schurmann}*{Section 2.2}).

One can define the notion of constructible sheaves of $R$-modules as in the complex analytic case. In particular, a sheaf $\cF$ of $R$-modules is constructible with respect to a stratification $\cS$ if the restriction of $\cF$ to each stratum is locally constant. The analogs of the assertions with regards to the derived category of constructible sheaves and $\cS$-constructible sheaves from \ref{subsubsec:constructible}
continue to hold in this setting.\\

\subsection{Sheaves on Simplicial schemes and spaces.}\label{subsec:sheavessimplical}
Let $\cX_{\bullet}$ be a simplicial topological space. A sheaf on $\cX_{\bullet}$ consists of the following data: 
\begin{enumerate}
\item A sheaf $\cF_k$ on $\cX_k$.
\item For a morphism $\phi: [k] \rightarrow [k']$ in $\Delta$, there is a map
$$\phi^*\cF_{k} \rightarrow \cF_{k'}.$$
\end{enumerate}
Moreover, the morphisms in the second part above are required to be compatible with composition of morphisms. A sheaf of $R$-modules is a sheaf as above where each $\cF_n$ is a sheaf of $R$-modules, and the morphisms in part (2) above are required to be morphisms of sheaves of $R$-modules. We denote by $\Sh(\cX_{\bullet},R)$ the category of sheaves of $R$-modules. We note that this is an abelian category, and denote by $\rD^b(\cX_{\bullet},R)$ the corresponding derived category.

If $\cX_{\bullet}$ is a simplicial scheme, then we define the category of constructible sheaves as before by requiring each $\cF_k$ to be constructible. As before, given a morphism $f: \cX_{\bullet} \rightarrow \cY_{\bullet}$ we have natural push-forward and pull-back functors. If $\varepsilon: \cX_{\bullet} \rightarrow S$ is an augmentation, then one also has push-forward and pull-back functors:
$$R\varepsilon_*: \rD^b(\cX_{\bullet},R) \rightarrow \rD^+(S,R),$$
and
$$\varepsilon^*: \rD^b(S,R) \rightarrow \rD^b(\cX_{\bullet},R).$$
We note that these functors preserve constructible sheaves.

\begin{remark}
One can define constructible sheaves on a simplicial scheme in the \'{e}tale topology in an entirely analogous manner. Similar remarks apply to the o-minimal setting.
\end{remark}

We end this subsection by recording the following simplicial analog of Lemma \ref{lem:projectionisconstant}.

\begin{lemma}\label{lem:projectionisconstantsimplicial}

Let $p: \cA_{\bullet} \times X \rightarrow X$ and $q: \cA_{\bullet} \times X \rightarrow \cA_{\bullet}$ be the natural projection maps, where $\cA_{\bullet}$ is a simplicial scheme with each $A_j$ proper. Let $\cG_{\bullet}$ be a simplicial sheaf on $\cA_{\bullet}$. Then $Rp_*(q^*\cG_{\bullet}) \isom \underline{R\Gamma(\cA_{\bullet},\cG_{\bullet})}$. In particular, $R^ip_*(q^*\cG_{\bullet})$ is the constant sheaf on $X$ given by $\underline{R^i\Gamma(\cA_{\bullet},\cG_{\bullet})}$.

\end{lemma}
\begin{proof}
This is a consequence of Lemma \ref{lem:projectionisconstant} applied to each $A_j$.
\end{proof}

\begin{remark}\label{rem:projectionisconstantinetale}
We note that the previous result is also true in the \'{e}tale and topological setting (and in particular the o-minimal setting).
\end{remark}

\subsection{Cohomological Descent}
Consider an augmented simplicial topological space
$\varepsilon: \frX_{\bullet} \rightarrow S$. We shall assume that all our topological spaces are locally compact and Hausdorff. The morphism $\varepsilon$ is said to be a \emph{morphism of cohomological descent} if the adjunction map 
$$ \cF \rightarrow R\varepsilon_* \varepsilon^*\cF$$
is an isomorphism for all $\cF \in \rD^b(S,R)$ (i.e. the derived category of abelian sheaves on $S$). We say that $\varepsilon$ is a morphism of universal cohomological descent if it is a morphism of cohomological descent after base change along any morphism $S' \rightarrow S$. The following well known theorem will be useful in the following

\begin{theorem}\label{thm:properdescent}
Let $\varepsilon: X \rightarrow S$ be a proper surjective map. Then the resulting morphism $\varepsilon: \cosk(X/S) \rightarrow S$ is a morphism of universal cohomological descent.
\end{theorem}

\begin{remark}\label{rem:etalecohomdescent}
\begin{enumerate}
\item If $f: X \rightarrow S$ is a proper surjective morphism of schemes over $\bbC$, then the induced map on complex analytic spaces is a proper surjective map of topological spaces. Moreover, the underlying analytic space of a separated scheme of finite type is locally compact and Hausdorff (in the complex topology).
\item One can also define morphisms of (universal) cohomological descent in the setting of sheaves in the \'{e}tale topology on a scheme in the same manner. The analog of Theorem~\ref{thm:properdescent} is also true in the setting of schemes and the \'{e}tale topology. In particular, if $\varepsilon: X \rightarrow S$ is a proper surjective map of schemes, then the resulting morphism $\varepsilon: \cosk(X/S) \rightarrow S$  is a morphism of universal cohomological descent. Here we consider the adjunction morphism in the derived category of sheaves in the \'{e}tale topology on $S$.
\item Note that the underlying topological space 
of a 
closed definable set
in an o-minimal
expansion of $\mathbb{R}$ 
is locally compact and Hausdorff. Moreover, a proper map in the category of definable spaces is proper as a map of topological space. In particular, the previous Theorem is also applicable in that setting.
\end{enumerate}
\end{remark}
 
\subsection{Properties of correspondences}
\label{subsec:properties:correpondence}
Given two correspondences $[D;A,B]$ and $[E;B,C]$ we may {\it compose} the correspondences to obtain a correspondence $[D \times_B E;A,C]$. Note that in this case one has a natural commutative diagram:
$$
\xymatrix{
  & & \ar[dr] \ar[dl]D \times_B E & & \\
  & D \ar[dr] \ar[dl] &   & E \ar[dr] \ar[dl]  & \\
  A  & & B & & C .
}
$$
\begin{remark}
With notation as above, note that the composition of two finite correspondences is again a finite correspondence over $A \times B \times C$. In order to see this, note that 
$D \times_B E \isom (D \times C) \times_{A \times B \times C} (A \times E)$. On the other hand, $D \times C \rightarrow A \times B \times C$ (respectively $A \times E \rightarrow A \times B \times C$) is finite. Consequently, if $B$ is proper, then $D \times_B E \rightarrow A \times C$ is proper.
\end{remark}

Given two schemes $A,B$, the {\it full product} correspondence is by definition the correspondence $[A \times B;A,B]$ (with $A \times B \rightarrow A$ and $A \times B \rightarrow B$ the natural projection maps). Note that the composition $[A \times B;A,B]$ and $[E;B,C]$ is the correspondence
$$
\xymatrix{
 &  A \times E    \ar[dr] \ar[dl]& \\  
A &   & C
}
$$
where the arrow on the left is the natural projection map, and the arrow on the right is projection to $E$ followed by the correspondence map $\pi_C: E \rightarrow C$. 

Given a closed subscheme $B_0 \hookrightarrow B$, we denote by $[D_{B_0}; A, B_0]$ the resulting `base-changed' correspondence. In particular, $D_{B_0}:= D \times_B B_0$ with the natural induced maps to $A$ and $B_0$. The morphism $D_{B_0} \rightarrow B_0$ is the projection map, and the morphism $D_{B_0} \rightarrow A$ is the composite of the projection to $D$ and the given morphism $D \rightarrow A$. Since $D_{B_0} = D \times_{A \times B} (A \times B_0)$, it follows that this is a finite correspondence. Similarly, if $A_0 \hookrightarrow A$ is a closed subscheme, then we denote by $[\leftindex_{A_0}{D}; A_0,B]$ the resulting base change to $A_0$. Finally, we also consider the base changes: $[\leftindex_{A_0}{D}_{B_0}; A_0,B_0]$. In the following, we shall only be concerned with base change along closed (reduced) subschemes of dimension 0. In particular, our $A_0$ and $B_0$ will be a finite set of disjoint points. 

Let $[D;A,B]$ and $[E;B,C]$ be finite correspondences, and consider $0$-dimensional closed (reduced) subschemes $C_0 \subset C$ and $A_0 \subset A$ as above. In this setting, we may compose the given correspondences and base change the resulting correspondence to obtain a correspondence: $[\leftindex_{A_0}{D\times_B E}_{C_0}; A_0,C_0]$. Note that this correspondence may also be obtained by base changing first, and then composing the resulting correspondences. Recall there is a natural structure morphism $D \times_B E \rightarrow B$, and therefore one has a natural morphism $\leftindex_{A_0}{D\times_B E}_{C_0} \rightarrow B$. 

\begin{lemma}\label{lem:corrcohomdescent}
 With notation and hypotheses as above, the natural morphism 
 \[
 \leftindex_{A_0}{D\times_B E}_{C_0} \rightarrow B
 \]
 is proper. In particular, if $I$ denotes the scheme theoretic image, then the resulting map 
 $$\cosk(\leftindex_{A_0}{D\times_B E}_{C_0}/I) \rightarrow I$$ is a morphism of cohomological descent in the \'{e}tale topology, and on the underlying complex analytic spaces.
 \end{lemma}
 \begin{proof}
By the remarks above, we may assume that $A = A_0$ and $C =C_0$ are zero dimensional i.e. a collection of closed points. In this case, the morphisms $B \times C_0 \rightarrow B$ and $A_0 \times B \rightarrow B$ are both finite. It follows that $D \rightarrow B$ and $E \rightarrow B$ are both finite morphisms, and therefore $D \times_B E \rightarrow B$ is a finite morphism. In particular, it is a proper morphism. The descent assertion now follows from \ref{thm:properdescent} and \ref{rem:etalecohomdescent}.
 \end{proof}   

In the preceding lemma, since the morphism is proper, its set theoretic image is closed and equal to the underlying closed subset of its scheme theoretic image $I$.

\begin{remark}\label{rem:corrcohomdefinabledescent}
In the setting of definable spaces, we shall simply work with closed subspaces i.e. take correspondences where $D$ is a closed definable subspace of $A \times B$. In this case, the analog of the previous lemma is also true.
\end{remark}

\subsection{Spectral sequence associated to a simplicial sheaf}
\label{subsec:spectral}
In this section, we recall some standard spectral sequences associated with simplicial spaces and sheaves, and some of their standard properties. We also prove a key result on the behavior of kernels between a morphism of such spectral sequences.

Given a simplicial space $\cX_{\bullet}$, and an abelian sheaf $\cF_{\bullet}$, one has a $1$-st quadrant cohomological spectral sequence:
$$
\rE_1^{j,i} = \rH^{i}(X_j,\cF_j) \Rightarrow \rH^{i+j}(\cX_{\bullet},\cF_{\bullet}).
$$
We shall denote this spectral sequence by $\rE(\cX_{\bullet},\cF_{\bullet})$. We note that if $\cF_{\bullet}$ is a sheaf of $R$-modules, then the cohomology groups above are $R$-modules, and the spectral sequence is a spectral sequence of $R$-modules.

One also has a `local' version of this spectral sequence. Let $a: \cX_{\bullet} \rightarrow S$ be an augmented simplicial space, and $\cF_{\bullet}$ a sheaf on $\cX_{\bullet}$. Let $a_j: X_j \rightarrow S$ denote the resulting structure maps. In this setting, we have a $1$-st quadrant cohomological spectral sequence:
$$
\rE_1^{j,i} = \rR^{i}a_{j,*}\cF_j \Rightarrow \rR^{i+j}a_*\cF_{\bullet}.
$$
We shall denote this spectral sequence by 
$\rE(\cX_{\bullet}/S,\cF_{\bullet})$. Moreover, we will denote by $\rE_r^{j,i}(\cF_{\bullet})$ the $\rE_{r}^{j,i}$ terms of this spectral sequence (where we allow $r = \infty$). We note that if $\cF_{\bullet}$ is a constructible sheaf of $R$-modules, then this is a spectral sequence of constructible sheaves of $R$-modules. 

\begin{remark}

\begin{enumerate}
\item If we take $a: \cX_{\bullet} \rightarrow *$ to be the canonical structure map to the point (i.e. the morphism to $\Spec(k)$ or $\Spec(\bbC)$), then $\rE(\cX_{\bullet}/*,\cF_{\bullet})$ is the spectral sequence $\rE(\cX_{\bullet},\cF_{\bullet}).$
\item The spectral sequences above are compatible under base change $S' \rightarrow S$ if $a_p$ is a proper morphism. This is a consequence of the proper base change which is true in all of our settings: constructible sheaves on complex analytic spaces, constructible sheaves in the \'{e}tale topology, and constructible sheaves on definable spaces. This implies in particular, that the `fibers' over a point $s \in S$ of the local spectral sequence is given by the former `global' spectral sequence. 
\end{enumerate}
\end{remark}

The following lemma will be useful in what follows.

\begin{lemma}\label{lem:Sconsttosubobject}
Let $f: \cA \rightarrow \cB$ be a morphism of constructible sheaves on $X$, where $\cA$ is $\cS$-constructible and $\cB$ is a subobject of an $\cS$-constructible sheaf. Then $\ker(f)$ is $\cS$-constructible. Similarly, the image of $f$ is an $\cS$-constructible sheaf.
\end{lemma}
\begin{proof}
Let $\cH$ be an $\cS$-constructible sheaf, and $i: \cB \hookrightarrow \cH$ a monomorphism. Then $\ker(f) = \ker(i \circ f)$. The result follows since the category of $\cS$-constructible sheaves is an abelian category. The same argument can also be applied to the image.
\end{proof}

\begin{lemma}\label{lem:Sconstruct1}
Let $a: \cX_{\bullet} \rightarrow S$ be an augmented simplicial space, and $\cF_{\bullet}$ a constructible sheaf on $\cX_{\bullet}$. Suppose that $R^ia_{j,*}\cF_{j}$ is $\cS$-constructible for all $j \leq k$. Then the following holds:
\begin{enumerate}
\item $\rE_{r}^{j,i}(\cF_{\bullet})$ is $\cS$-constructible for all $j \leq k -r +1$;
\item $\rE_r^{j,i}(\cF_{\bullet})$ is a subobject of an $\cS$-constructible sheaf for all $k-r+1 < j \leq k$.
\end{enumerate}
\end{lemma}
\begin{proof}
We prove this by induction on $r$. If $r =1$, then the assertion holds given our assumptions. In general, $\rE_{r+1}^{j,i}(\cF_{\bullet})$ is the homology of the complex:
$$\rE_r^{j-r,i+r-1}(\cF_{\bullet}) \rightarrow \rE_r^{j,i}(\cF_{\bullet}) \rightarrow \rE_r^{j+r,i-r+1}(\cF_{\bullet}).$$
Suppose $j < k-r+1$. Then $j +r \leq k$ and by the induction hypothesis, it follows that the first two terms are $\cS$-constructible and the last term is a subobject of an $\cS$-constructible sheaf. By Lemma \ref{lem:Sconsttosubobject}, the kernel of the second arrow is $\cS$-constructible. Since the image of the first arrow is $\cS$-constructible, it follows that the corresponding homology is $\cS$-constructible. Suppose now that $k - r+1 \leq j \leq k$. In this case, $j - r \leq k-r$. It follows that the left-most term is $\cS$-constructible. The middle term is a sub-object of an $\cS$-constructible sheaf. In particular, the kernel of the right arrow is a subobject of an $\cS$-constructible sheaf. Moreover, by Lemma \ref{lem:Sconsttosubobject}, the image of the left arrow is $\cS$-constructible. It follows that the homology is a subobject of an $\cS$-constructible sheaf.
 
\end{proof}

The spectral sequences above are also functorial in the following sense. Suppose we are given a morphism of simplicial spaces (over $S$) $F: \cX_{\bullet} \rightarrow \cY_{\bullet}$, a sheaf $\cG_{\bullet}$ on $\cY_{\bullet}$, and a morphism $\phi: F^{*}(\cG_{\bullet}) \rightarrow \cF_{\bullet}$.
This data induces a natural morphism of spectral sequences
$$
\rE(\cY_{\bullet}/S,\cG_{\bullet}) \rightarrow \rE(\cX_{\bullet}/S,\cF_{\bullet})
$$
where the maps on the $\rE_1^{p,q}$ terms and on the abutments are the natural induced maps on cohomology (via functoriality). In the following, we will be interested in gleaning information about the kernels of the induced morphisms 
$$ \rR^{i}a_{\bullet,*}(\cG_{\bullet}) \rightarrow \rR^{i}b_{\bullet,*}(\cF_{\bullet})$$
of sheaves on $S$ (where $a_{\bullet}: \cY_{\bullet} \rightarrow S, b_{\bullet}: \cX_{\bullet} \rightarrow S$ are the augmentation maps) from the corresponding kernels of the induced maps on the $\rE_1$ terms of the aforementioned spectral sequences. In particular, the following theorem will be key in the following.
\begin{theorem} \label{thm:spectralseqinput}
Let $F: \cX_{\bullet} \rightarrow \cY_{\bullet}$, $\cF_{\bullet}$, $\cG_{\bullet}$, $a_{\bullet}$, $b_{\bullet}$ and $\phi$ be as above. Suppose that:
\begin{enumerate}
\item \(R^i a_{\bullet,*}\mathcal G_\bullet\) is a constant sheaf (with finitely generated stalks) on \(S\) for all \(i\leq k\).
\item \(\mathcal S\) is a stratification of \(S\) such that
\(Rb_{p,*}\mathcal F_p\) is \(\mathcal S\)-constructible for all \(p\leq k\).
\end{enumerate}
Then
\[
\cK_i:=\ker\bigl(R^i a_{\bullet,*}\mathcal G_\bullet
       \longrightarrow R^i b_{\bullet,*}\mathcal F_\bullet\bigr)
\]
is \(\mathcal S\)-constructible for all \(i\leq k\). Moreover, its restriction
to each connected component of a stratum is a constant subsheaf of the source.
\end{theorem}

\begin{proof}
Let \(F_{i,*}\colon R^i a_{\bullet,*}\mathcal G_\bullet
\to R^i b_{\bullet,*}\mathcal F_\bullet\) denote the induced morphism on abutments of the morphism of spectral sequences
\[
\rE(\cY_\bullet/S,\mathcal G_\bullet)
\longrightarrow \rE(\cX_\bullet/S,\mathcal F_\bullet).
\]
We prove the theorem for \(i=k\). The proof for \(i<k\) is the same.
The abutment \(R^k b_{\bullet,*}\mathcal F_\bullet\) has a filtration
\(F^p\) with \(F^0=R^k b_{\bullet,*}\mathcal F_\bullet\),
\(F^{k+1}=0\), and
\[
F^p/F^{p+1}\cong \rE_\infty^{p,k-p}(\mathcal F_\bullet).
\]
First, note that \(\rE_{k+2}^{p,k-p}(\cF_\bullet)=\rE_\infty^{p,k-p}(\cF_\bullet)\). To see this, note that for \(0\leq p\leq k\) and \(r\geq k+2\), the incoming differential \(d_r\)
starts in column \(p-r<0\), and the outgoing differential \(d_r\) lands in row
\(k-p-r+1<0\). Since this is a first quadrant spectral sequence, both differentials vanish. We may now apply Lemma \ref{lem:Sconstruct1} to conclude that
each of these graded pieces is a subobject of an
\(\mathcal S\)-constructible sheaf.

Fix a connected component \(T\) of a stratum. We show by descending
induction on \(p\) that every morphism \(u\colon M\to F^p|_T\) from a
constant sheaf \(M\) has constant kernel. This is immediate for \(p=k+1\).
For the induction step, let
\[
N :=\ker\bigl(M\xrightarrow{u}F^p|_T
                \longrightarrow (F^p/F^{p+1})|_T\bigr).
\]
The quotient \((F^p/F^{p+1})|_T\) embeds in a local system on \(T\).
Composing with this inclusion leaves the kernel equal to \(N\), so
\(N\) is constant by Remark \ref{rem:localsystemanalytic} (and its etale analog \ref{rem:localsystemsetale}). Since the restriction of \(u\) to \(N\) factors
through \(F^{p+1}|_T\), by the induction hypothesis,
\(\ker(u)=\ker(N\to F^{p+1}|_T)\) is constant.

Taking \(p=0\) and \(u=F_{k,*}|_T\), now gives that
\(\cK_k|_T\) is constant (with finitely generated stalks). In particular, \(\cK_k\) is
\(\mathcal S\)-constructible.
\end{proof}

\section{Complexity of stratifications}
\label{sec:stratification}
In this section, we explain some VC (co)density results and apply them to prove the following key result on lengths of stratifications. In turn, this result will be one of the key ingredients
in the proof of 
the complex and \'{e}tale versions of Theorem~\ref{thm:main}.

Let $X,Z$ be $k$-schemes of finite type with $k$ an algebraically closed field, and fix an integer $p \geq 0$. Given a closed point $z \in Z$ and a subscheme 
$Y \subset X \times Z$, we will denote by $Y_{z}$
the pull-back of the diagram
\[
\xymatrix{
& Y \ar[d]_{\pi_Z} \\
z \ar[r] & Z,
}
\]
where
$\pi_Z:Y \rightarrow Z$ is the restriction of the canonical projection
$X \times Z \rightarrow Z$ to  $Y$. In particular, we have $Y_{z} = Y \times_{X \times Z} (X \times z)$. Since the base change of a subscheme is a subscheme, we have a natural immersion $Y_z \hookrightarrow X \times z \isom X$, and we may identify $Y_z$  with its image in $X$. In particular, in the following we consider $Y_z$ as a subscheme of $X$.\footnote{Given two schemes $X$ and $Z$ over $k$, $X \times_k Z$ denotes the fiber product over $k$.} 

Let $\cY := (Y_0,\ldots,Y_{r})$ be a filtration of $X \times Z$ by subschemes where $$Y_0 \supset Y_1 \supset \cdots \supset Y_r.$$
We may apply the aforementioned base change construction to obtain a filtration $\cY_z := (Y_{0,z},\ldots,Y_{r,z})$ of $X$.

\begin{remark}
With notation as above, note that one can associate a stratification to $\mathcal{Y}$. If $\cS(Y)$ is the associated stratification, then the length of $\cS(Y) \leq r+1$.
\end{remark}
 
Given a closed point $z = (z_0,\ldots,z_n) \in Z^{n+1}$ and an ordered subset $J = \{j_1,\ldots,j_l\} \subset \{0,\ldots,n\}$, we denote by $z^{(J)} = (z_{j_1},\ldots,z_{j_{l}}) \in Z^{l}$ the resulting closed point. With this notation, we have the following key result:

\begin{theorem}
\label{thm:vcdensityapp}
For $1 \leq i \leq p+1$, 
let $\mathcal{Y}_i := (Y_{i,0} \supset \cdots \supset Y_{i,r_i})$, be a filtration of $X \times Z^i$ by closed subschemes.
Then there  exists a constant $C > 0$ depending on 
$\mathcal{Y}_1,\ldots,\mathcal{Y}_{p+1}$,
such that for all $n >0$, and closed points
$z^{[n]} = (z_0,\ldots, z_n) \in Z^{n+1}$, there exists a stratification $\cS$
of $X$ with connected strata such that 

\begin{enumerate}
    \item The length of $\cS$ is bounded by 
    \[
C \cdot (n+1)^{(p+1)\dim X},
\]
\item $\cS$ refines the stratifications $\cS(\cY_{i,z^{(J)}})$ of  $X$ which are canonically 
associated to the filtrations $\cY_{i,z^{(J)}}$ for all $i$ and ordered subsets $J \subset \{0,\ldots,n\}$ of order $i$.
\end{enumerate}

\end{theorem}

The proof of Theorem~\ref{thm:vcdensityapp} relies on the following theorem (Theorem~\ref{thm:Hilbert}) which gives upper bounds on the number of realizable $0/1$-patterns for a family of hypersurfaces in affine space (i.e. VC-codensity bounds). We begin with some notation. Let $X$ be a $k$-scheme of finite type
and $\mathcal{S}$ a finite set of closed subschemes of $X$. Given  $\sigma \in \{0,1\}^{\mathcal{S}}$, we set
\[
\RR(\sigma) := \{x \in X(k) \mid \forall Y \in \mathcal{S},  x \in Y \; \Leftrightarrow \; \sigma(Y) =1\}. 
\]
Note that $\RR(\sigma)$ can be viewed as the set of rational points of a locally closed subset of $X$. More precisely, it is given by the rational points of \[\left( \bigcap_{\substack{
Y \in \mathcal{S} \\ \sigma(Y) =1}} Y \right ) \setminus \left ( \bigcup_{\substack{Y\in\mathcal{S} \\ \sigma(Y) =0}} Y \right ) .\]

\hide{
We denote
\[
\Sigma(X,\mathcal{S}) := \{\sigma \in \{0,1\}^{\mathcal{S}} \mid \RR(\sigma) \neq \emptyset\},
\]
\[
N(X,\mathcal{S}) := \sum_{\sigma \in \{0,1\}^{\mathcal{S}}} b_0(\RR(\sigma)),
\]
and 
\[
N'(X,\mathcal{S}) := \card(\Sigma(X,\mathcal{S})).
\]
Clearly, 
\[
N'(X,\mathcal{S}) \leq N(X,\mathcal{S}).
\]

The following theorem in the case $X = \mathbb{A}^m$ is the main result of  (\cite{RBG01}, see also \cite{BPR09}). However, the proof of this result in the aforementioned papers do not extend to the case when $X \neq \mathbb{A}^m$. We give below a cohomological approach to such bounds (which is valid for general $X$).
}

We denote
\[
\Sigma(X,\mathcal{S}) := \{\sigma \in \{0,1\}^{\mathcal{S}} \mid \RR(\sigma) \neq \emptyset\},
\]
and 
\[
N(X,\mathcal{S}) := \sum_{\sigma \in \{0,1\}^{\mathcal{S}}} b_0(\RR(\sigma)).
\]

\begin{theorem}
    \label{thm:Hilbert}
    Suppose $X \subset \mathbb{A}^m$ is a closed subscheme and fix an integer $D \geq 0$. Then there exists $C = C(X,D) > 0$, such that for any $N > 0$, $F_1,\ldots,F_N \in k[X_1,\ldots,X_m]_{\leq D}$, 
    \[
    N(X,\mathcal{S}) \leq C \cdot N^{\dim X}
    \]
    where $\mathcal{S} = \{Y_1,\ldots,Y_N\}$, and $Y_i$ is the intersection of the hypersurface defined by  $F_i = 0$ with $X$.
    \end{theorem}

\begin{proof}
 Follows immediately from \cite[Theorem 2.1]{Basu-Parida} which gives a bound on 
 \[
 \sum_{\sigma \in \{0,1\}^{\mathcal{S}}} \deg(\overline{\RR(\sigma)}),
 \]
 noting that $\deg(\overline{\RR(\sigma)})$ 
 is the sum of the degrees of all reduced irreducible components of
 $\overline{\RR(\sigma)}$, including components of smaller dimension. Then $\deg(\overline{\RR(\sigma)})$ is at least the number of irreducible components of $\overline{\RR(\sigma)}$. Moreover, for each
 $\sigma \in \{0,1\}^{\mathcal{S}}$,
 $\RR(\sigma)$ is open in $\overline{\RR(\sigma)}$, and each non-empty 
 intersection of $\RR(\sigma)$ with an irreducible component of  
 $\overline{\RR(\sigma)}$ is connected. Thus 
 the number of irreducible components of $\overline{\RR(\sigma)}$ is an upper bound on $b_0(\RR(\sigma))$. The theorem now follows
 from \cite[Theorem 2.1]{Basu-Parida}.
\end{proof}

\subsection{Proof of Theorem \ref{thm:vcdensityapp}}
We first prove a Lemma in the setting that $X \subset \mathbb{A}^m, Z \subset \mathbb{A}^{\ell}, p =0,$ and the filtration $\cY$ is a single closed subscheme $Y \subset X \times Z$.

\begin{lemma}
\label{lem:vc}
With notation as above, there exists a constant $C > 0$ which depends only on  the embedding
$Y \hookrightarrow X \times Z$, such that for any $n >0$,  $z^{[n]} = (z_0,\ldots,z_n) \in Z^{n+1}$, 
\[
N(X,\mathcal{S}(z^{[n]})) \leq C \cdot (n+1)^{\dim X},
\]
where $\mathcal{S}(z^{[n]}) = (Y_{z_0},\ldots,Y_{z_n})$.
\end{lemma}

\begin{remark}
With notation as in the above lemma, note that  
\[
(\RR(\sigma))_{\sigma \in \Sigma(X,\mathcal{S}(z^{[n]}))}
\]
gives a stratification of $X$ of length 
$\card(\Sigma(X,\mathcal{S}(z^{[n]})))$. 
In particular, after passing to connected components of the realizations, the lemma immediately implies Theorem \ref{thm:vcdensityapp} in this special case.
\end{remark}

\begin{proof}[Proof of Lemma~\ref{lem:vc}]
    Suppose that $X \subset \mathbb{A}^m = \Spec\; k[X_1,\ldots,X_m], Z \subset \mathbb{A}^\ell = \Spec\; k[Z_1,\ldots,Z_\ell]$. Then $Y$ is an affine subscheme of $\mathbb{A}^m \times \mathbb{A}^\ell$ and is defined by a finite number of polynomial equations (say) 
    \[
    F_1(X_1,\ldots,X_m,Z_1,\ldots,Z_\ell)  =  \cdots = F_M(X_1,\ldots,X_m,Z_1,\ldots,Z_\ell) = 0.
    \]

    For $z^{[n]} = (z_0,\ldots,z_n) \in Z^{n+1} $, let 
    \[
    \mathcal{F} = (F_{i,j}(X_1,\ldots,X_m,z_j))_{1 \leq i \leq M, 0 \leq j \leq n}.
    \]
    For a fixed $j$, the collection of $M$ polynomials $\{F_{i,j}\}$ defines $Y_{z_j} \subset X$.
    
    Note that each realization of $0/1$-pattern of the tuple of
    fibers $Y_{z_j}, 0 \leq j \leq n$, is a union of the realizations of some subset of  
    $0/1$-patterns on the $M(n+1)$ hypersurface defined by the polynomials 
    $F_{i,j}$.
    It now follows from Theorem~\ref{thm:Hilbert} 
    that
    \begin{eqnarray*}
    N(X,\mathcal{S}(z^{[n]})) &\leq&  
    C \cdot (n+1)^{\dim X}
    \end{eqnarray*}
    where $C$ depends on the embedding 
    $Y \hookrightarrow X \times Z$ but is clearly independent of $z^{[n]}$. Note that the degrees of $F_{i,j}$ do not depend on the point $z^{[n]}.$ 
\end{proof}

In the following lemma we use the same notation as in Theorem~\ref{thm:vcdensityapp}.

\begin{lemma}
\label{lem:vcdensityapp}
    There exists a constant $C > 0$, which depends on $p$ and 
    the closed immersions $\phi_{i,j}: Y_{i,j} \hookrightarrow X \times Z^i, 1 \leq i \leq p+1$, having the following property:
    for each $n>0$, and  $z^{([n])} = (z_0,\ldots,z_n) \in Z^{n+1}$, 
   \[ 
   N(X,\mathcal{S}) \leq C \cdot (n+1)^{(p+1)\dim X},
   \]
   where
    \[
    \mathcal{S} = \left(Y_{i,j,z^{(J)}}\right)_{\substack{
    1 \leq i \leq p+1,
    0 \leq j \leq r_i,\\
    J \subset [0,n], 
    \card(J) = i
    }
    }
    \]
    and 
    for an ordered subset $J \subset [n]$,  $z^{(J)} = (z_j)_{j \in J}$. 
\end{lemma}
\begin{proof}
We reduce to the affine case as follows. Choose finite open affine coverings $(X_\alpha)_{\alpha \in A}, (Z_\beta)_{\beta \in B}$ of
$X,Z$ by affine subschemes using the fact that $X,Z$ are of finite type. For $\beta^{(i)}= (\beta_1,\ldots,\beta_i) \in B^i$, we denote by 
$Z_{\beta^{(i)}}$ the subscheme 
$Z_{\beta_1} \times \cdots \times Z_{\beta_i}$ of $Z^i$. Next, observe that  
\[
\left(Z_{\beta^{(i)}}\right)_{\beta^{(i)} \in B^i}
\]
is a finite affine covering of $Z^i$.
Set
\[
\widetilde Z =
\coprod_{\substack{1\le i\le p+1,\;0\le j\le r_i\\
                   \boldsymbol\beta\in B^i}}
Z_{\boldsymbol\beta}.
\]
This is affine, being a finite disjoint union of affine schemes.
For each \(\alpha\in A\), define the closed subscheme
\[
\widetilde Y_\alpha =
\coprod_{\substack{1\le i\le p+1,\;0\le j\le r_i\\
                   \boldsymbol\beta\in B^i}}
\bigl(Y_{i,j}\cap(X_\alpha\times Z_{\boldsymbol\beta})\bigr)
\subset X_\alpha\times\widetilde Z.
\]
Each component retains its label
\((i,j,\boldsymbol\beta)\).

For each \(z^{(J)}\), choose a chart
\(Z_{\boldsymbol\beta}\) containing it.
The fibre at its copy in the component labelled by
\((i,j,\boldsymbol\beta)\) is
\[
Y_{i,j,z^{(J)}}\cap X_\alpha.
\]
Hence all the tests in \(\mathcal S|_{X_\alpha}\) are fibers of this
single family, with at most
\[
N =
\sum_{i=1}^{p+1}(r_i+1)\binom{n+1}{i}
\le C_0(n+1)^{p+1}
\]
selected parameters, where \(C_0\) is independent of the sample.
Applying Lemma~\ref{lem:vc} gives
\[
N(X_\alpha,\mathcal S|_{X_\alpha})
\le C_\alpha N^{\dim X_\alpha}
\le C'_\alpha(n+1)^{(p+1)\dim X}.
\]
Finally,
\[
N(X,\mathcal S)
\le \sum_{\alpha\in A}
N(X_\alpha,\mathcal S|_{X_\alpha})
\le C(n+1)^{(p+1)\dim X},
\]
since the affine cover of \(X\) is fixed and finite. 
\end{proof}

\begin{proof}[Proof of Theorem~\ref{thm:vcdensityapp}]
We follow the same notation as in Lemma~\ref{lem:vcdensityapp}.
Lemma~\ref{lem:vcdensityapp} implies that there exists $C > 0$
depending only on the filtrations  $\mathcal{Y}_1,\ldots,\mathcal{Y}_{p+1}$
such that for all $z^{[n]} \in Z^{n+1}$,
$N(X,\mathcal{S}(z^{[n]})) \leq C \cdot (n+1)^{(p+1)\dim X}$.
Now observe that 
\[
(\RR(\sigma))_{\sigma \in \Sigma(X,\mathcal{S}(z^{[n]}))}
\]
is a stratification of $X$ of length 
$\card(\Sigma(X,\mathcal{S}(z^{[n]})))$.
Passing to connected components of the realizations gives the required stratification, whose length satisfies the same bound by the connected component estimate in Lemma~\ref{lem:vcdensityapp}.
\end{proof}

\subsection{O-minimal version}
In order prove our main theorems in the o-minimal setting, we will also need the o-minimal version of Theorem~\ref{thm:vcdensityapp}. We state and prove the required result in this section. Fix an o-minimal expansion of $\mathbb{R}$. Recall, we define stratifications for definable sets of this structure following the same lines as in the case of schemes (\ref{subsubsec:o-minimalsheaves}).  

In particular, a {\it stratification} of a definable set $X$ is a finite collection of locally closed (in the Euclidean topology) subsets  $S_i \subset X$ ($i \in  I$),
such that $X = \coprod S_i$. We refer to $|I|$ as the {\it length} of the stratification. We refer to the $S_i$ as strata (or stratum) and use the notation $\cS$ to denote the collection of strata $S_i$.
Suppose $X = X_0 \supsetneq X_1 \supsetneq X_2 \supsetneq \cdots \supsetneq X_n \neq \emptyset$ is a filtration of $X$ by closed subsets. Then we may associate a canonical stratification of length $n +1$. We set $S_i := X_{i} \setminus X_{i+1}$ for $0 \leq i < n$ and $S_n = X_n$. Note that each $S_i$ is locally closed. A {\it refinement} of a stratification $\cS$ is a stratification $\cS'$ such that each stratum $S_i \in \cS$ is a union of strata $S_i' \in \cS'$. 

Now let $X,Z$ be definable sets, and $p \geq 0$. For $1 \leq i \leq p+1$, 
let $\mathcal{Y}_i := (Y_{i,0},\ldots,Y_{i,r_i})$ be an $r_i + 1$-tuple of closed definable sets  $Y_{i,j} \hookrightarrow X \times Z^i$, such that
\[
X\times Z^i = Y_{i,0} \supset Y_{i,1} \supset \cdots \supset Y_{i,r_i}.
\]

For $z^{(i)} \in Z^i$, and a definable subset 
$Y \subset X \times Z^i$, we will denote by $Y_{z^{(i)}}$
the pull-back of the diagram
\[
\xymatrix{
& Y \ar[d]_{\pi^{(i)}_{Z^i}} \\
\{z^{(i)}\} \ar[r] & Z^i,
}
\]
where 
$\pi^{(i)}_{Z^i}:Y \rightarrow Z^i$ is the restriction of the canonical projection
$X \times Z^i \rightarrow Z^i$ to  $Y$
(and identify $Y_{z^{(i)}}$  with its image in $X$).

\begin{theorem}
\label{thm:vcdensityapp:om}
For $1 \leq i \leq p+1$, 
let $\mathcal{Y}_i := (Y_{i,0} \supset \cdots \supset Y_{i,r_i})$, be a filtration of $X \times Z^i$ by closed definable sets.
Then there  exists a constant $C > 0$ depending on 
$\mathcal{Y}_1,\ldots,\mathcal{Y}_{p+1}$,
such that for all $n >0$, and
$z^{[n]} = (z_0,\ldots, z_n) \in Z^{n+1}$, there exists a stratification 
of $X$ with connected strata, of length bounded by 
\[
C \cdot (n+1)^{(p+1)\dim X},
\]
which refines the stratifications of  $X$ which are canonically 
associated to the filtrations
\[
\left(\mathcal{Y}_{i,z^{(J)}} = (Y_{i,0,z^{(J)}} \supset \cdots \supset Y_{i,r_i,z^{(J)}})
\right)_{1 \leq i \leq p+1, z^{(J)} \in Z^{(J)}}.
\]
Here for each $i$, $J$ ranges over ordered subsets of $\{0,\ldots,n\}$ of cardinality $i$.
\end{theorem}

The proof of Theorem~\ref{thm:vcdensityapp:om} is similar to the proof of 
Theorem~\ref{thm:vcdensityapp} (in fact simpler since we do not have to reduce to the affine case).
We begin with some notation which mirrors the scheme case.

\hide{
Let $X$ be a definable set and
and $\mathcal{S}$ a finite set of closed definable subsets of $X$. Given  $\sigma \in \{0,1\}^{\mathcal{S}}$, we set
\[
\RR(\sigma) := \{x \in X \mid \forall Y \in \mathcal{S},  x \in Y \; \Leftrightarrow \; \sigma(Y) =1\}. 
\]

We set
\[
\Sigma(X,\mathcal{S}) := \{\sigma \in \{0,1\}^{\mathcal{S}} \mid \RR(\sigma) \neq \emptyset\},
\]
and 
\[
N(X,\mathcal{S}) := \card(\Sigma(X,\mathcal{S})).
\]
}
Let $X$ be a definable set
and $\mathcal{S}$ a finite set of closed definable subsets of $X$. Given  $\sigma \in \{0,1\}^{\mathcal{S}}$, we set
\[
\RR(\sigma) := \{x \in X \mid \forall Y \in \mathcal{S},  x \in Y \; \Leftrightarrow \; \sigma(Y) =1\}, 
\]
\[
\Sigma(X,\mathcal{S}) := \{\sigma \in \{0,1\}^{\mathcal{S}} \mid \RR(\sigma) \neq \emptyset\},
\]
and
\[
N(X,\mathcal{S}) := \sum_{\sigma \in \{0,1\}^{\mathcal{S}}} b_0(\RR(\sigma)).
\]

Now, suppose that $X , Z$ be definable sets and  $Y \subset X \times Z$ 
a closed definable subset. For $z \in Z$, we denote by
$Y_z$ the pull-back of the diagram
\[
\xymatrix{
& Y \ar[d]_{\pi_Z} \\
\{z\} \ar[r] & Z,
}
\]
where 
$\pi_Z:Y \rightarrow Z$ is the restrictions of the canonical projection
$X \times Z \rightarrow Z$ to  $Y$.
We will identify $Y_z$ with its canonical immersion in $X$ (given by the 
restriction of the canonical projection $X \times Z \rightarrow X$ to
$Y_z$)
and consider $Y_z$ as a definable subset of $X$.

\begin{lemma}
\label{lem:vc:om}
There exists a constant $C > 0$ which depends only on  the embedding
$Y \hookrightarrow X \times Z$, such that for any $n >0$,  $z^{n} = (z_0,\ldots,z_n) \in Z^{n+1}$, 
\[
N(X,\mathcal{S}(z^{n})) \leq C \cdot (n+1)^{\dim X},
\]
where $\mathcal{S}(z^{n}) = (Y_{z_0},\ldots,Y_{z_n})$.
\end{lemma}

\begin{proof}
In the case $X,Z$ are closed, the lemma follows directly from
\cite[Theorem 2.2]{Basu10}. 
In the general case,
choose fixed finite cell partitions $X=\coprod_\alpha C_\alpha$ and $Z=\coprod_\beta D_\beta$, and definable homeomorphisms
\[
\mathbb{R}^{e_\alpha}\simeq C_\alpha,
\qquad \mathbb{R}^{f_\beta}\simeq D_\beta.
\]
The pullback of $Y\cap(C_\alpha\times D_\beta)$ is closed in $\mathbb{R}^{e_\alpha+f_\beta}$. 
We then apply \cite[Theorem 2.2]{Basu10} on each $\mathbb{R}^{e_\alpha}$, combining the finitely many parameter-cell families by disjoint tags. For a sample of size $n+1$, this gives $C_\alpha (n+1)^{e_\alpha}$ components on $C_\alpha$. Therefore
\[
\sum_\sigma b_0(\operatorname{Reali}_X(\sigma))
\leq\sum_\alpha C_\alpha (n+1)^{e_\alpha}
\leq C(n+1)^{\dim X}.
\]
\end{proof}

In the following lemma we use the same notation as in Theorem~\ref{thm:vcdensityapp:om}.

\begin{lemma}
\label{lem:vcdensityapp:om}
    There exists a constant $C > 0$, which depends on $p$ and 
    the closed immersions $\phi_{i,j}: Y_{i,j} \hookrightarrow X \times Z^i, 1 \leq i \leq p+1$, having the following property:
    for each $n>0$, and  $z^{[n]} = (z_0,\ldots,z_n) \in Z^{n+1}$, 
   \[ 
   N(X,\mathcal{S}) \leq C \cdot (n+1)^{(p+1)\dim X},
   \]
   where
    \[
    \mathcal{S} = \left(Y_{i,j,z^{(J)}}\right)_{\substack{
    1 \leq i \leq p+1,
    0 \leq j \leq r_i,\\
    J \subset [0,n], 
    \card(J) = i
    }
    }
    \]
    and 
    for ordered subsets $J \subset [n]$,  $z^{(J)} = (z_j)_{j \in J}$. 
\end{lemma}
\begin{proof}
Let
\[
\widetilde Z=
\coprod_{\substack{1\le i\le p+1\\0\le j\le r_i}}Z^i,
\]
retaining the labels \((i,j)\) for the component, and let \(\widetilde Y\subset X\times\widetilde Z\) be the closed definable family whose restriction to the component labeled \((i,j)\) is \(Y_{i,j}\).

Each member \(Y_{i,j,z^{(J)}}\) of \(\mathcal{S}\) is a fiber of this single family. The length $N$ of $\mathcal{S}$ is at most 
\[
N=\sum_{i=1}^{p+1}(r_i+1)(n+1)^i
\le C_0(n+1)^{p+1},
\]
where \(C_0\) depends only on the fixed filtrations and \(p\).

Applying Lemma~\ref{lem:vc:om} to \(\widetilde Y\) shows that the total number of connected components of the realizations is bounded by
\[
C_1 N^{\dim X}
\le C(n+1)^{(p+1)\dim X}.
\]
\end{proof}

\hide{
\begin{proof}
Apply inequality \eqref{lem:vc:om} with the tuple
\[
\mathcal{S}(z^{([n])}) = \left(Y_{i,j,z^{(J)}}\right)_{\substack{
    1 \leq i \leq p+1,
    0 \leq j \leq r_i,\\
    J \subset [0,n], 
    \card(J) = i
    }
    }
\]
noting that the length $N$ of the tuple $\mathcal{S}(z^{[n]})$
$
\leq  \sum_{i=1}^{p+1} \binom{n+1}{i}.
$
We obtain that
\begin{eqnarray*}
N(X,\mathcal{S}(z^{[n]}))  &\leq& C' \cdot N^{\dim X} \\
&\leq & C' \cdot \left( \sum_{i=1}^{p+1} \binom{n+1}{i}\right)^{\dim X}\\
&\leq & C' \cdot \left( \sum_{i=1}^{p+1} \left( (n+1)^i\right)\right) ^{\dim X} \\
&\leq& 
C' \cdot (p+1) \cdot n^{(p+1)\dim X} \\
&\leq & C \cdot n^{(p+1)\dim X},
\end{eqnarray*}
where $C = C' \cdot (p+1)$.

This completes the proof of the lemma.
\end{proof}
}

\begin{proof}[Proof of Theorem~\ref{thm:vcdensityapp:om}]
We follow the same notation as in Lemma~\ref{lem:vcdensityapp:om}.
Lemma~\ref{lem:vcdensityapp:om} implies that there exists $C > 0$
depending only on the filtrations  $\mathcal{Y}_1,\ldots,\mathcal{Y}_{p+1}$
such that for all $z^{[n]} \in Z^{n+1}$,
$N(X,\mathcal{S}(z^{[n]})) \leq C \cdot (n+1)^{(p+1)\dim X}$.
Now observe that 
\[
(\RR(\sigma))_{\sigma \in \Sigma(X,\mathcal{S}(z^{[n]}))}
\]
is a stratification of $X$ of length 
$\card(\Sigma(X,\mathcal{S}(z^{[n]})))$.
Passing to connected components of the realizations gives the required stratification, whose length satisfies the same bound by the connected component estimate in Lemma~\ref{lem:vcdensityapp:om}.
\end{proof}

\section{Proofs of Theorems~\ref{thm:informal} and \ref{thm:main}
}
\label{sec:proof:main:complex}
Before proceeding to the general setting of the main theorems, we consider some special cases. These cases are not needed for the general argument, but are included here in order to clarify some of the main ideas appearing in the proof of the Theorem. 

Below we work simultaneously in the following three settings:
\begin{enumerate}
\item The category of definable spaces and constructible sheaves in some 
o-minimal expansion of $\mathbb{R}$.
\item The category of schemes over the complex numbers, and constructible sheaves in the complex analytic topology.
\item The category of schemes over an arbitrary algebraically closed field $k$, and constructible ($\ell$-adic) sheaves in the \'{e}tale topology (for some fixed prime $\ell$ prime to the characteristic of the field $k$).
\end{enumerate}
Since essentially the same argument (except for the proofs of the relevant bounds on number of strata proved in the previous section) works in all three settings, we will simply refer to our objects as spaces with the understanding that we are in one of the aforementioned settings. Moreover, we shall simply work with singular cohomology with rational coefficients with the understanding that these should be replaced with \'{e}tale cohomology in the third setting above.

\begin{remark}
Note that below we use the terminology of proper morphisms. In the o-minimal setting, 
$Y$ is assumed to be compact, 
and so this is automatic. Similarly, the terminology closed immersion should be understood as closed immersion of schemes in the $\mathrm{ACF}$  case, and as a closed inclusion in the o-minimal setting.
\end{remark}

\subsection{The case when $n=0$}
In this case, the theorem asserts the existence of a constant $C$ (depending only on the correspondences $H$ and $\Lambda$) such that set of kernels 
$$
    \ker\left(
\mathrm{H}^p\left( \Lambda_{z},\bbQ\right)
    \longrightarrow
    \mathrm{H}^p\left( \leftindex_{x}{\Lambda}_{z},\bbQ\right)
    \right)
$$
is bounded by this constant. One can give a direct proof of this fact as follows.
\begin{proof}
Consider the induced morphisms $H \times_Y \Lambda \rightarrow Y$ and  $H \times \Lambda \rightarrow Z$. Here the second map is given by the composition of $H \times \Lambda \rightarrow \Lambda \rightarrow Z$. One has a cartesian diagram:
$$
\xymatrix{
H \times_Y \Lambda_z \ar[r] \ar[d]^{p_z} & H \times_Y \Lambda \ar[d]^p \\
X = X \times z \ar[r]  & X \times Z. }
$$
Here the right vertical is given by the composition $H \times_Y \Lambda \rightarrow H \times \Lambda \rightarrow X \times Z$ where the second map is the correspondence map on each component.
Let $\cS$ be a stratification of $X \times Z$ such that $Rp_*\bbQ$ is $\cS$-constructible. Let $C(\cS)$ be the number of elements in this partition. By definition, this constant only depends on the given correspondences. Note that for each $z \in Z$ one has an induced partition of $X$. While this partition varies with $z$, the number of elements in these induced partitions is bounded by $C$. We now make the following observations:
\begin{enumerate}
\item The morphism $p$ is proper. In order to see this, first note that $\Lambda \rightarrow Z$ is proper since it is the composition of a closed immersion followed by the projection $Y \times Z \rightarrow Z$. The latter is proper since we assume $Y$ are proper. Similarly, $H \rightarrow X$ is proper. It follows that $H \times \Lambda \rightarrow X \times Z$ is proper. Finally, note that $H \times_Y \Lambda \rightarrow H \times \Lambda$ is a closed immersion and therefore proper. 
\item Since the base change of a proper morphism is proper, it follows that $p_z$ is a proper morphism. In particular, the stalks of the sheaf $R^ip_{z,*}\bbQ$ at $x \in X$ are given by the cohomology of the fibers: $\rH^i(\leftindex_{x}{\Lambda}_z,\bbQ)$.
\item By proper base change, note that $Rp_{z,*}\bbQ \isom Rp_*\bbQ|_{X \times z}$ is constructible with respect to the partition of $X$ given by pulling back $\cS$ along $X \times \{z\} \hookrightarrow X \times Z$. We denote the resulting stratification of $X$ by $\cS_{z}$, and note that its length is bounded by $C$. Passing to connected components and enlarging $C$, we may assume $\cS_z$ are connected.
\end{enumerate}

 On the other hand, one has a natural map $H \times_Y \Lambda \rightarrow X \times \Lambda$, and in particular a morphism $f:H \times_Y \Lambda_z \rightarrow X \times \Lambda_z$. If $q_z: X \times \Lambda_z \rightarrow X$ denotes the natural projection map, then by \ref{lem:projectionisconstant} $R^iq_{z,*}\bbQ$ is a constant local system on $X$ given by $\mathrm{H}^i\left( \Lambda_{z},\bbQ\right)$. By functoriality, the commutative diagram
 \[
\xymatrix{
H \times_Y \Lambda_z \ar[rr]^{f}\ar[rd]^{p_z} && X \times \Lambda_z \ar[ld]^{q_z} \\
& X &
}
\]
induces a morphism $R^iq_{z,*}\bbQ \rightarrow R^ip_{z,*}\bbQ$ of $\cS_{z}$-constructible sheaves where the domain is a constant local system. For a given point $x \in X$, the induced morphism of stalks is just the usual pull-back map on cohomology:
$$\mathrm{H}^i\left( \Lambda_{z},\bbQ\right) \rightarrow \mathrm{H}^i\left( \leftindex_{x}{\Lambda}_{z},\bbQ\right).$$

On each stratum we obtain a morphism from a constant sheaf to a local system. By \ref{prop:kernelisindep} the restriction maps 
on cohomology 
$$ 
\mathrm{H}^i\left( \Lambda_{z},\bbQ\right)
    \longrightarrow
    \mathrm{H}^i\left( \leftindex_{x}{\Lambda}_{z},\bbQ\right)
    $$
    
    have the same kernel (namely the global sections of the kernel of the corresponding local systems) for all $x$ in a fixed stratum. In particular, the number of such sub-modules, as we vary $x \in X$, is bounded by $C$.

\end{proof}

\begin{remark}
We note that the aforementioned proof will also work for cohomology with coefficients in an commutative ring $R$.
\end{remark}

\subsection{The case when $n =1$}
In order to explicate our method, we consider the case $n=1$. In particular, we now consider points $(z_0,z_1) \in Z \times Z$. 

\begin{proof}
Let $Z_0: = (z_0,z_1) \in Z \times Z$. Let $\Lambda^{(2)} := \Lambda \times_Y \Lambda$. Note that this is equipped with a natural morphism to $Z \times Z$. Consider now the induced morphism $ p: H \times_Y \Lambda^{(2)} \rightarrow X \times Z \times Z$. Note that we have a diagram of 1-truncated semi-simplicial spaces:
$$
\xymatrix{
H \times_Y \Lambda^{(2)} \ar[r]^{f_1} \ar@<-.5ex>[d] \ar@<.5ex>[d] & X \times Z \times Z \ar@<-.5ex>[d] \ar@<.5ex>[d]\\
H \times_Y \Lambda \ar[r]^{f_0} & X \times Z.}
$$
We denote the truncated semi-simplicial space on the left by 
$H \times_Y \Lambda_{\bullet}$ and the one on the right by $X \times Z_{\bullet}$.
Let $\cS_1$ be a stratification of $X \times Z \times Z$ such that $Rf_{1,*}\bbQ$ is constructible with respect to this stratification, and similarly let $\cS_0$ be a stratification of $X \times Z$ so that $Rf_{0,*}\bbQ$ is constructible with respect to $\cS_0$.
Given $(z_0,z_1) \in Z \times Z$, consider the semi-simplicial space
$$
\xymatrix{
X \times \{(z_0,z_1)\} \ar@<-.5ex>[r] \ar@<.5ex>[r] & X \times \{z_0\} \coprod X \times \{z_1\}
}.
$$
This semi-simplicial space is of course isomorphic to $X \rightrightarrows X \coprod X$. However, keeping track of the points in $Z$ gives a canonical map to the semi-simplicial space $X \times Z_{\bullet}$.
We may base change the aforementioned diagram along this map to obtain the semi-simplicial space (denoted $H\times_Y \Lambda_{(z_0,z_1),\bullet}$):
$$ 
\xymatrix{
H \times_Y (\Lambda_{z_0} \times_Y \Lambda_{z_1}) \ar@<-.5ex>[r] \ar@<.5ex>[r] &H \times_Y (\Lambda_{z_0} \coprod \Lambda_{z_1})}.
$$
Note that all the spaces above are augmented over $X$. Let $a_1: H \times_Y \Lambda^{(2)} \rightarrow X$
and $a_0: H \times_Y \Lambda \rightarrow X$ denote the corresponding augmentation maps. Recall, these are given by first projecting to $H$ and then considering the given map $H \rightarrow X$. We use the same notation for the corresponding augmentation maps for $H \times_Y (\Lambda_{z_0} \times_Y \Lambda_{z_1})$ and $H \times_Y (\Lambda_{z_0} \coprod \Lambda_{z_1})$. The constructions of \ref{subsec:sheavessimplical} and \ref{subsec:spectral} apply equally to (truncated) semi-simplicial spaces.
This gives rise, for each $(z_0,z_1)$, a spectral sequence $\rE(H\times_Y \Lambda_{(z_0,z_1),\bullet}/X, \bbQ)$:
$$\rE_1^{j,i} = R^ia_{j,*}\bbQ \Longrightarrow R^{i+j}a_{\bullet,*}\bbQ.$$
Here $a_{\bullet} $ is the augmentation map at the level of simplicial spaces. Since we are working with truncated semi-simplicial spaces, this spectral sequence is concentrated in two columns. In particular, this spectral sequence degenerates at the second page. We note that all terms are constructible sheaves on $X$. We make the following observations:
\begin{enumerate}
\item The stratification $\cS_0$ induces stratifications $\cS_{0,z_i}$ on 
$X$ via the embedding $X \times \{z_i\} \hookrightarrow X \times Z $, and similarly a stratification $\cS_{1,(z_0,z_1)}$ (via the embedding $X \times \{(z_0,z_1)\} \hookrightarrow X \times Z \times Z$).
\item Consider the augmentation $a_0: \coprod H \times_Y \Lambda_{z_i} \rightarrow \coprod X \times \{z_i\} \rightarrow X$, and let $a_{0,i}$ denote the two components of this morphism. The higher direct image is given by the direct sum $\oplus Ra_{0,i,*}\bbQ$. An easy consequence of proper base change shows that the components are constructible along $\cS_{0,z_i}$ (respectively). For example, consider the cartesian diagram:
$$
\xymatrix{
H \times_Y \Lambda_{z_i}  \ar[d] \ar[r]& H \times_Y \Lambda \ar[d] \\
X \times {z_i} \ar[r] & X \times Z.}
$$
Since the right vertical arrow is proper, we may apply proper base change to conclude that 
$$ Ra_{0,i,*}\bbQ \isom Rf_{0,*}\bbQ|_{X \times {z_i}}.$$

\item Similarly, the higher direct image along the augmentation
$a_1: H \times_Y \Lambda^{(2)}_{(z_0,z_1)} \rightarrow X \times (z_0,z_1) \rightarrow X$ is constructible along the stratification given by $\cS_{1,(z_0,z_1)}$. 
\item Now consider a refinement $\cS$ of these three stratifications as in Theorem \ref{thm:vcdensityapp} (or Theorem \ref{thm:vcdensityapp:om} in the o-minimal setting). It follows that all the terms on the second page of the spectral sequence above (and hence the abutments) are $\cS$-constructible.
\item We now apply the previous constructions to $H = X \times Y$, and look at the resulting spectral sequence. In this case, $H \times_Y \Lambda = X \times \Lambda$, $H \times_Y \Lambda^{(2)} = X \times \Lambda^{(2)}$,
the semi-simplicial space $H \times_Y \Lambda_{(z_0,z_1),\bullet} = 
X \times \Lambda_{(z_0,z_1),\bullet}$
and similarly for the base changed versions considered above. Since the augmentation maps are just projection maps to $X$, the direct images to $X$ are constant local systems on $X$ (by Lemma \ref{lem:projectionisconstant}). Moreover, the spectral sequence $\rE(X \times \Lambda_{(z_0,z_1),\bullet}/X,\bbQ)$ maps to the one considered above (by functoriality since $H \hookrightarrow X \times Y$). 
\item By Lemma~\ref{lem:projectionisconstantsimplicial} (and Remark~\ref{rem:projectionisconstantinetale}), the map in (5) induces a morphism from the constant sheaf associated with \(\rH^i(\Lambda_{(z_0,z_1),\bullet},\mathbb Q)\) to \(R^i a_{\bullet,*}\mathbb Q\). The latter is \(\mathcal S\)-constructible by (4). Remark \ref{rem:localsystemanalytic} and Proposition \ref{prop:kernelisindep} therefore show that its kernel is a fixed subspace of \(\rH^i(\Lambda_{(z_0,z_1),\bullet},\mathbb Q)\) on each stratum.
\item Finally, consider taking fiber along points of $x \in X$ of the abutments. By proper base change, we compute the cohomology of the semi-simplicial space obtained by base changing to $x \in X$. For our initial semi-simplicial space this is:
$$ \leftindex_x{H} \times_Y (\Lambda_{z_0} \times_Y \Lambda_{z_1}) \rightrightarrows \coprod \leftindex_x{H} \times_Y \Lambda_{z_i}.$$
By cohomological descent for the proper surjection
\[
{}_x\Lambda_{z_0}\amalg{}_x\Lambda_{z_1}
\longrightarrow
{}_x\Lambda_{z_0}\cup{}_x\Lambda_{z_1},
\]
and comparison with the usual (ordered) Čech complex, the above semi-simplicial space computes the cohomology of \({}_x\Lambda_{z_0}\cup{}_x\Lambda_{z_1}\).
On the other hand, in the case of $H = X \times Y$, the corresponding fibers are the cohomology of $\bigcup \Lambda_{z_i}$. Since the kernels are constant local systems on each stratum, it follows that the kernels of these restriction maps on the fibers at $x$ are the same for all $x$ in a stratum. In particular, the set of kernels is bounded by the length of $\cS$.

\end{enumerate}

This completes the proof. 
\end{proof}

\begin{remark}
We note that in the special case considered above, the cohomological degree does not play a role in the desired bound. This is since already the truncated semi-simplicial space computes all cohomology groups. 
\end{remark}

\subsection{Proof in the general case}
The strategy of proof in general is similar to the case considered in the previous section except that in general our simplicial spaces are no longer $1$-truncated, and the spectral sequences intervening in the proof no longer degenerate at the second page. Fix $p \geq 0$. \\

We begin by first introducing the relevant simplicial objects. Let $z = (z_0,\ldots,z_n) \in Z^{n+1}$. We consider the following simplicial spaces:
\begin{enumerate}
\item $H \times_Y \cosk(\Lambda/Y)$. Recall, $(H \times_Y \cosk(\Lambda/Y))_k = H \times_Y (\Lambda \times_Y \cdots \times_Y \Lambda)$, where $\Lambda$ appears $k+1$ times in the product on the right. 
\item $X \times \cosk(Z/*)$, where the $k$-th terms is $X \times Z \times \cdots \times Z$ with $k+1$ factors of $Z$.
\item Let $Z_0 := \coprod z_i$, and consider $\cosk(Z_0/*)$. Note that there is natural (inclusion) morphism $$X \times \cosk(Z_0/*) \hookrightarrow X \times \cosk(Z/*).$$
\item Let $f_{\bullet}: H \times_Y \cosk(\Lambda/Y) \rightarrow X \times \cosk(Z/*)$ denote the map of simplicial spaces given by the composite of the morphisms 
$$H \times_Y \cosk(\Lambda/Y) \rightarrow H \times \cosk(\Lambda/Y) \rightarrow X \times \cosk(Z/*),$$
where the second map is given by the maps $H \rightarrow X$ and $\cosk(\Lambda/Y) \rightarrow \cosk(Z/*)$. The latter map results from functoriality of the coskeleton. Since $Y$ is proper, the projection map $Y \times Z \rightarrow Z$ is also proper. Moreover, closed immersions are proper, and therefore $\Lambda \rightarrow Z$ is proper. It follows that the resulting map on coskeletons $\cosk(\Lambda/*) \rightarrow \cosk(Z/*)$ is a proper morphism of simplicial spaces.\footnote{A morphism of simplicial spaces $f_{\bullet}: \cX \rightarrow \cY$ is proper if each $f_n: \cX_n \rightarrow \cY_n$ is a proper morphism.} On the other hand, $\cosk(\Lambda/Y) \hookrightarrow \cosk(\Lambda/*)$ is a closed immersion. Since $H \rightarrow X$ is also proper (given that $Y$ is proper), it follows that $f_{\bullet}$ is a proper morphism of simplicial spaces.
\item We have the analogous construction $g_{\bullet}: X \times \cosk(\Lambda/Y) \rightarrow X \times \cosk(Z/*).$ Again, this is a proper morphism (it is simply the identity on the first component). We note that this construction can be obtained by setting $H = X \times Y$ in the preceding construction i.e. we take the identity correspondence.
\item It follows by functoriality that we have a natural commutative diagram of simplicial spaces:
$$
\xymatrix{
H \times_Y \cosk(\Lambda/Y) \ar[r]^{\pi_{\bullet}} \ar[dr]^{f_{\bullet}} &X \times \cosk(\Lambda/Y) \ar[d]^{g_{\bullet}} \\
   & X \times \cosk(Z/*).
}
$$
The horizontal map is induced by the inclusion $H \hookrightarrow X \times Y$.
\item Consider $(z_0,\ldots,z_n) \in Z^{n+1}$. We may and do assume that the $z_i$ are distinct points. Let $Z_0:= \{z_0, \ldots, z_n\} \subset Z$ denote the corresponding closed subspace, and consider the coskeleton $\cosk(Z_0/*)$. Recall, we can describe the $k$-th simplices as follows:
$$\cosk(Z_0/*)_k = Z_0 \times \cdots \times Z_0 = \coprod_{j \subset \{0,\ldots,n\}^{k+1}} (z_{j_0},\ldots,z_{j_k})$$
where $j = (j_0,\ldots,j_k)$. By functoriality, there is a natural closed immersion
$X \times \cosk(Z_0/*) \hookrightarrow X \times \cosk(Z/*)$. For each level $k$, it is just given by inclusion of the points.
\item Consider now the higher direct images $Rf_{i,*}\bbQ$ on $X \times Z^{i+1}$ for all $0 \leq i \leq p$ and choose stratifications $\cS_i$ of $X \times Z^{i+1}$ so that $Rf_{i,*}\bbQ$ is constructible with respect to $\cS_i$. We now choose as stratification $\cS$ of $X$ as in Theorems \ref{thm:vcdensityapp:om} (in the o-minimal case) and \ref{thm:vcdensityapp} (in the complex analytic and \'{e}tale cases). 
We note that the length of this stratification is bounded by 
\[
C \cdot (n+1)^{(p+1) \cdot \dim(X)}
\]
where $C$ only depends on the (locally closed) immersions $\cS_{i,j} \hookrightarrow X \times Z^{i+1}$ and $p$. Here $\cS_{i,j}$ are the strata of $\cS_i$. We note that we can always find stratifications $\cS_i$ so that it is a stratification associated to a filtration. Note that these in turn only depend on the correspondences $[H;X,Y]$ and $[\Lambda;Y,Z]$. 
\item We consider the commutative diagram of simplicial spaces above base-changed to $Z_0$:
$$
\xymatrix{
H \times_Y \cosk(\Lambda_{Z_0}/Y) \ar[r] \ar[dr] & X \times \cosk(\Lambda_{Z_0}/Y) \ar[d] \\
&  X \times \cosk(Z_0/*)
}
$$
Note that these are all augmented over $X$ via the natural projection maps. In particular, we have a commutative diagram of augmented simplicial spaces:
$$
\xymatrix{
H \times_Y \cosk(\Lambda_{Z_0}/Y) \ar[r] \ar[dr]^{b_{\bullet}} & X \times \cosk(\Lambda_{Z_0}/Y) \ar[d]^{a_{\bullet}} \\
&  X}
$$
Let $\cX := H \times_Y \cosk(\Lambda_{Z_0}/Y)$ and $\cY:= X \times \cosk(\Lambda_{Z_0}/Y)$.
The diagram above gives a morphism of spectral sequences:
$$\rE(\cY/X,\bbQ) \rightarrow \rE(\cX/X,\bbQ).$$
For \(\mathbf j=(j_0,\ldots,j_k)\in\{0,\ldots,n\}^{k+1}\), let
\[
\iota_{\mathbf j}:X\longrightarrow X\times Z^{k+1},
\qquad x\longmapsto(x,z_{j_0},\ldots,z_{j_k}).
\]
By proper base change,
\[
Rb_{k,*}\mathbb Q
\simeq
\bigoplus_{\mathbf j\in\{0,\ldots,n\}^{k+1}}
\iota_{\mathbf j}^{*}Rf_{k,*}\mathbb Q.
\]
Since each \(\Lambda_z\hookrightarrow Y\) is a closed immersion, deleting repetitions or permuting entries leaves the corresponding intersection unchanged. Thus (8) implies that every summand, and hence \(Rb_{k,*}\mathbb Q\), is \(\mathcal S\)-constructible for \(0\le k\le p\). We note that same argument applies to show that $Ra_{k,*}\bbQ$ is $\cS$-constructible for all $0 \leq k \leq p$. Moreover, by Lemma \ref{lem:projectionisconstant} in the complex setting and the analogous assertions in \ref{subsec:constructibleetale}, \ref{subsubsec:o-minimalsheaves} for the \'{e}tale and o-minimal settings, $Ra_{k,*}\bbQ$ is a constant sheaf. Similarly, by Lemma \ref{lem:projectionisconstantsimplicial} and Remark \ref{rem:projectionisconstantinetale}, $Ra_{\bullet,*}\bbQ$ is a constant sheaf with stalk at $x \in X$ given by $R\Gamma(\cosk(\Lambda_{Z_0}/Y),\bbQ)$. Let $\cK_i := ker(R^ia_{\bullet,*}\bbQ \rightarrow R^ib_{\bullet,*}\bbQ)$. By Theorem~\ref{thm:spectralseqinput}, $\cK_i$ is $\cS$-constructible which is constant on each stratum, for all $0 \leq i \leq p$.

\item On the other hand, by cohomological descent (Lemma \ref{lem:corrcohomdescent}),
\[
\rH^i(\cosk(\Lambda_{Z_0}/Y),\bbQ) = \rH^i(\bigcup_{j=0}^{n} \Lambda_{z_j},\bbQ).
\]
Similarly, by proper base change and cohomological descent, we note that 
$$
(R^ib_{\bullet,*}\bbQ)_x = \rH^{i}(\leftindex_{x}{H} \times_Y \cosk(\Lambda_{Z_0}/Y),\bbQ) = \mathrm{H}^i\left(\bigcup_{j=0}^{n} \leftindex_{x}{\Lambda}_{z_{j}},\bbQ\right)
$$
It follows that at the level of stalks, the morphism 
$$R^pa_{\bullet,*}\bbQ \rightarrow R^pb_{\bullet,*}\bbQ$$
is given by the restriction map on cohomology:
$$
\mathrm{H}^p\left(\bigcup_{j=0}^{n} \Lambda_{z_{j}},\bbQ\right)
    \longrightarrow
    \mathrm{H}^p\left(\bigcup_{j=0}^{n} \leftindex_{x}{\Lambda}_{z_{j}},\bbQ\right)
$$
The kernel of this map is \(\mathcal K_{p,x}\). By Theorem~\ref{thm:spectralseqinput}, it is independent of \(x\) within each stratum of \(\mathcal S\). Consequently, the number of kernels is bounded by the length of \(\mathcal S\), hence by \(C(n+1)^{(p+1)\dim X}\), by (8). This completes the proof.
\end{enumerate}

\begin{proof}[Proof of Theorem~\ref{thm:informal}]
For 
$\bbM$
an o-minimal expansion of $\mathbb{R}$, 
or a model of one of the theories
$\mathrm{ACF}(0)$, or $\mathrm{ACF}(p)$,
the theorem follows from 
the o-minimal, complex as well as the characteristic $0$ \'{e}tale, and the characteristic $p$ \'{e}tale versions of Theorem~\ref{thm:main}, respectively.
In the case of 
models of
$\mathbf{T} = \mathrm{RCF}$, the theorem follows
from 
the o-minimal version of Theorem~\ref{thm:main}
using a standard argument (that we omit) 
involving effective semi-algebraic triangulations 
and the Tarski-Seidenberg transfer principle.
\end{proof}

\section{Tightness}
\label{sec:tightness}

We give examples showing that the exponent $(p+q) \dim X$ is optimal in the semialgebraic setting for every $p \geq 0$ and $q \geq 1$.
In the algebraic settings, the examples establish sharpness in
degree zero for order one, and in degree one for order one with $\dim X = 1$.

\begin{example}[Degree zero]
Let $m\geq1$ and $X=Y=Z=\bbP^m$ over an algebraically closed field.
Let $H\subset X\times Y$ be the point--hyperplane incidence variety
defined by the homogeneous equation $\sum_{i=0}^m X_iY_i=0$, and let $\Lambda\subset Y\times Z$ be the diagonal.
For $n$ points $z_1,\ldots,z_n$ whose dual hyperplanes are in
general position, the restriction kernels in degree
zero correspond exactly to the traces of hyperplanes on this
sample.  Every subset of at most $m$ of the dual hyperplanes
has a nonempty intersection away from all the remaining
hyperplanes, and no $m+1$ of them have a common point.
Thus the number of kernels is
\[
\sum_{i=0}^{m}\binom ni.
\]
For fixed $m$ this is $\Theta(n^m)$, proving sharpness of the
exponent $(p+1)\dim X=m$ when $p=0$.  The same argument works
with rational singular cohomology over $\bbC$ or 
with  \'{e}tale cohomology. 
\end{example}

\begin{example}[Degree one over an algebraically closed field]
Let $X=\bbP^1$, $Y=\bbP^2\times\bbP^2$, and
$Z=\check{\bbP}^{\,2}\times\check{\bbP}^{\,2}$.
For $(L,M) \in Z$, let $\Lambda_{(L,M)} = L \times M \subset Y$ be the test fiber.
Fix a point $o\in\bbP^2$ and let $(L_x)_{x\in X}$ be the pencil
of lines through $o$.
For $x \in X$, 
let $\leftindex_x{H}=L_x\times\bbP^2$ be the candidate fiber. The correspondences
$[\Lambda;Y,Z]$ and $[H;X,Y]$ defined above are
clearly closed.

For $n\geq 2$, choose lines $L_1,\ldots,L_n$ with no three
concurrent, with $o\notin\bigcup_iL_i$, and such that the
points $p_{ij}=L_i\cap L_j$ for $i<j$ lie on distinct members
of the pencil.  These are nonempty open general-position
conditions.  Choose $2n$ lines
$M_1,M'_1,\ldots,M_n,M'_n$ in the second plane, with no three
concurrent.  Sample the $2n$ test fibers
$L_i\times M_i,L_i\times M'_i$, and write
\[
A=\bigcup_{i=1}^n L_i\times(M_i\cup M'_i),\qquad
V=\bigcup_{i=1}^n(M_i\cup M'_i).
\]
The projection $\pi_2:A\to V$ is proper, its fibers are connected
and are either a projective line or the union of two projective
lines meeting at a point.  Their first cohomology is zero.
Proper base change and the degree-one part of the Leray spectral
sequence give the natural isomorphism
\[
\pi_2^*:\HH^1(V,\bbQ)\xrightarrow{\ \sim\ } \HH^1(A,\bbQ).
\]

Let $x_{ij}$ be the pencil member through $p_{ij}$.
For $x$ different from all the $x_{ij}$, the intersection
$A\cap \leftindex_x{H}$ is a disjoint union of copies of $M_i\cup M'_i$,
so its first cohomology is zero.
At $x=x_{ij}$, exactly the $i$-th and $j$-th copies merge,
and the only nonzero degree-one summand is
\[
\HH^1(V_{ij},\bbQ),\qquad V_{ij}=M_i\cup M'_i\cup M_j\cup M'_j.
\]
Under the natural projection isomorphisms, restriction from
$A$ to $A\cap H_{x_{ij}}$ becomes restriction from $V$ to
$V_{ij}$.

These restrictions have distinct kernels.
The closed-cover spectral sequence for $V$ identifies
$\HH^1(V,\bbQ)$ naturally with the first cohomology of the complete
graph on its $2n$ line components: the components and their
pairwise intersections are connected, the components have
$\HH^1=0$, and there are no triple intersections.
The four-line subunion corresponds to the complete subgraph
on $\{M_i,M'_i,M_j,M'_j\}$.
Inside the cycle space of the complete graph, the cycle spaces
of these four-vertex subgraphs are distinct.  Indeed, for two
different pairs, choose a triangle in the first four-vertex
set using a vertex absent from the second.  Its edge chain
cannot belong to the second cycle space.
Their annihilators, which are precisely the restriction
kernels, are therefore distinct over every coefficient field.
This proves
\[
\nu^1_{\mathcal X,\mathcal Z}(2n)
 \geq\binom n2\geq \tfrac1{16}(2n)^2 .
\]
The graph calculation is equally valid for singular and
\'{e}tale cohomology with the stated coefficients.
\end{example}

\begin{example}[Semialgebraic sharpness in every bidegree]
\label{ex:sharp-real}
First take $\dim X=1$ and $p\geq1$.
Let $Y_0=[0,1]^{p+1}$, and let its test family consist of all
coordinate slices $\{y_a=u\}$, with the direction $a$ as a
finite tag and $u\in[0,1]$.  For $1\leq i<q$, let
$Y_i=[0,1]$ with the family of singleton tests.
All parameter spaces and all correspondences are compact
and semialgebraic.

Choose positive real numbers $c_1,\ldots,c_{p+q}$ linearly
independent over $\bbQ$.  Let $h$ be the linear function with
these coefficients on $\prod_{i=0}^{q-1}Y_i$, and 
let
\[
H_t=\{y:h(y)\leq t\},\qquad
t\in X=[0,c_1+\cdots+c_{p+q}].
\]
In each coordinate direction of $Y_0$, choose $n\geq2$
distinct rational slice levels in $(0,1)$.  This gives
$m=(p+1)n$ test members and a grid $G\subset Y_0$ of
dimension $p$.
In each other factor, sample $m$ distinct rational points.
The product of the sampled unions is a disjoint union of
$m^{q-1}$ copies of $G$.

The complement of $G$ in $\bbR^{p+1}$ has
$(n-1)^{p+1}$ bounded rectangular components.
For a fixed copy and a cutoff $t$, a bounded component remains
bounded in the complement of $G\cap H_t$ precisely when the
height of its upper corner, including the fixed coordinates
of that copy, is at most $t$.
Indeed, in that case its whole boundary is retained.
Otherwise a missing boundary point of height greater than
$t$ connects the component to the upper halfspace and hence
to the unbounded component.
Alexander duality 
therefore gives
\[
\dim \HH^p(G\cap H_t,\bbQ)
 =\card\{\text{bounded grid cells with upper-corner height }\leq t\}
\]
on each copy.
All these upper-corner heights, across all copies, are distinct,
by rationality of the coordinates and the choice of the $c_a$.
The restriction maps from the full sampled union are surjective
in degree $p$: both source and closed target have dimension at
most $p$, as in the top-degree argument of
Lemma~\ref{lem:unique}.
The dimensions of their kernels consequently take at least
\[
(n-1)^{p+1}m^{q-1}+1=\Omega(m^{p+q})
\]
different values.  For $p=0$, use singleton tests in all $q$
interval factors instead; the distinct heights of the $m^q$
grid points give $m^q+1$ kernels directly.

For parameter dimension $d\geq1$, replace each $Y_i$ by a
disjoint union of $d$ tagged copies of the corresponding space.
Use $X=[0,c_1+\cdots+c_{p+q}]^d$.
On the product component where all tags equal $a$, use the
cutoff $t_a$; on components with unequal tags, let the candidate
fiber be empty.  These conditions define a closed semialgebraic
correspondence.  Distribute the tests equally among the $d$
tags in every factor.
The kernels on the $d$ diagonal product components vary
independently as direct summands; the off-diagonal kernel
summands are fixed.  Thus, along the sequence of these sample
sizes $M = d m$, there are
$\Omega(M^{(p+q)d})$ distinct kernels.
Together with Theorem~\ref{thm:main:general}, this proves
sharpness of the exponent $(p+q)d$ in the semialgebraic setting
for every $p,q,d$.
In particular, the case $p=1,q=1$ is the planar rectangular
grid example, where the first Betti number takes all values
from $0$ to $(n-1)^2$.
\end{example}

\section{Proofs of 
Theorems~\ref{thm:informal:general} and \ref{thm:main:general}}
\label{sec:proof:main:general}
We prove a slightly more precise version of the theorem. Its proof follows the same strategy outlined in Section~\ref{subsec:method} but in a more condensed form.

\begin{theorem}[Precise stratified version]
\label{thm:main:complex:general'}
Fix $p\geq 0$ and $q\geq 1$, and let
\[
[\Lambda_i;Y_i,Z_i]\quad(0\leq i\leq q-1),
\qquad [H;X,Y_{0}\times\cdots\times Y_{q-1}]
\]
be correspondences satisfying the closedness and properness hypotheses
of Theorem~\ref{thm:main}, in any one of the 
settings considered
there.  Let $n \geq 1$, and for 
$Z_i'\subset_n Z_i$, let
\[
U_i=\bigcup_{z\in Z_i'}\Lambda_{i,z},\qquad
A=
\prod_{i=0}^{q-1}U_i,
\qquad A_x=A\cap\leftindex_x H.
\]
Then, there exist a constant $C>0$ and a closed correspondence
$[\Gamma;X,V]$, depending only on $p$ and
the correspondences 
\[
[\Lambda_i;Y_i,Z_i]\quad(0\leq i\leq q-1),
\qquad [H;X,Y_{0}\times\cdots\times Y_{q-1}]
\]
such that one can choose $v_1,\ldots,v_N\in V$, with
\[
N\leq C \cdot  n^{p+q},
\]
for which
\[
K_x=\ker\bigl(\HH^p(A,\bbQ)\longrightarrow
                         \HH^p(A_x,\bbQ)\bigr)
\]
is independent of $x$ on every connected component of every realization of the  $0/1$-patterns 
of
$\{\Gamma_{v_1},\ldots,\Gamma_{v_N}\}$.
\hide{
In the o-minimal setting, connected components are definably connected
components.  In the algebraic settings they are connected components
of the associated locally closed reduced schemes; over $\bbC$ these
agree with the connected components in the analytic topology.
The coefficient notation is interpreted as in Theorem~\ref{thm:main}.
}
\end{theorem}

\begin{proof}
For a tuple $\mathbf a=(a_0,\ldots,a_{q-1})$ of nonnegative integers, 
we denote
$|\mathbf a|=\sum_i a_i$ and let
\[
T_{\mathbf a}=
\prod_{i=0}^{q-1}Z_i^{a_i+1}.
\]
For
$\eta=
(z_{i,j})_{0\leq i<q,\,0\leq j\leq a_i}
\in T_{\mathbf a}$, let
\[
A_\eta=
 \prod_{i=0}^{q-1}\left(\bigcap_{j=0}^{a_i}\Lambda_{i,z_{i,j}}\right),
\qquad B_\eta(x)=A_\eta\cap\leftindex_x H.
\]
These are the fibers of a fixed proper map
\[
f_{\mathbf a}:\mathcal D_{\mathbf a}\longrightarrow X\times T_{\mathbf a},
\]
where $\mathcal D_{\mathbf a}$ is the closed incidence subspace of
$X\times T_{\mathbf a}\times\prod_{i=0}^{q-1}Y_i$
defined by the displayed conditions.  Properness follows from
properness of the $Y_i$.

For each $|\mathbf a|\leq p$, choose a finite stratification
$\mathcal S_{\mathbf a}$ of $X\times T_{\mathbf a}$ with respect to
which $Rf_{\mathbf a,*}\bbQ$ is constructible.  Choose finitely many
closed subspaces
\[
D_{\mathbf a,1},\ldots,D_{\mathbf a,c_{\mathbf a}}
   \subset X\times T_{\mathbf a}
\]
whose $0/1$-patterns refine this stratification.  This is possible
because each stratum is a difference of two closed subspaces.
Define
\[
V=\coprod_{\substack{|\mathbf a|\leq p\\1\leq h\leq c_{\mathbf a}}}
 T_{\mathbf a},
\qquad
\Gamma=\coprod_{\substack{|\mathbf a|\leq p\\1\leq h\leq c_{\mathbf a}}}
 D_{\mathbf a,h}\ \subset X\times V.
\]
Thus $\Gamma$ is fixed before the finite sets $Z_i'$ are chosen.

Fix now 
the $Z_i'$.  For each $|\mathbf a|\leq p$,
include all fibers $(D_{\mathbf a,h})_\eta$ with
\[
\eta\in\prod_{i=0}^{q-1}(Z_i')^{a_i+1}.
\]
After discarding repetitions if necessary, their number satisfies
\begin{equation}
\label{eq:multi-cech-count}
N\leq\sum_{|\mathbf a|\leq p}c_{\mathbf a}n^{|\mathbf a|+q}
 \leq\left(\sum_{|\mathbf a|\leq p}c_{\mathbf a}\right)n^{p+q}.
\end{equation}
Let $S$ be a 
connected component of the
realization of a $0/1$-pattern of these fibers.
For every selected $\eta$, proper base change shows that the
cohomology sheaves of
\[
R(f_\eta)_*\bbQ,\qquad
f_\eta:\{(x,y):y\in B_\eta(x)\}\longrightarrow X,
\]
restrict to local systems on $S$.

Let $B=H\cap(X\times A)$ and let $f:B\to X$ be projection.
For each $i$, the closed subsets of $B$ defined by
$y_i\in\Lambda_{i,z}$, $z\in Z_i'$, form a finite cover of $B$.
Take the alternating augmented \v{C}ech resolution for each of
these $q$ covers, and tensor the resolutions over the coefficient
field.  This is a resolution of $\bbQ_B$: at a geometric stalk,
each augmented factor is the augmented cochain complex of the
nonempty simplex of covering members containing that point.

Its term in multidegree $\mathbf a$ is the direct sum of the
constant sheaves on the intersections obtained by choosing
$a_i+1$ distinct members of the $i$-th cover for every $i$,
extended by the corresponding closed immersions.  In particular,
these intersections are precisely the spaces $B_\eta$ above.

Applying the derived direct image gives a first-quadrant double
complex $\mathcal C^{r,b}$, where $r=|\mathbf a|$ is the total
\v{C}ech degree and $b\geq0$ is the resolution degree.  Its total
complex $\mathcal T$ represents $Rf_*\bbQ_B$.  The vertical
complex in column $r$ represents a direct sum of the complexes
$R(f_\eta)_*\bbQ$ with $|\mathbf a|=r$.  
The number of parameter entries in multidegree $\mathbf a$ is
$q+|\mathbf a|$.

Let $\mathcal T_{\leq p}$ be the quotient total complex obtained
by keeping the columns $0\leq r\leq p$ and setting the outgoing
horizontal differential from column $p$ equal to zero.
The discarded columns form a subcomplex $\mathcal T_{>p}$,
and
\[
0\longrightarrow\mathcal T_{>p}\longrightarrow\mathcal T
 \longrightarrow\mathcal T_{\leq p}\longrightarrow0
\]
is exact.  Since $\mathcal C^{r,b}=0$ for $b<0$, the complex
$\mathcal T_{>p}$ has no terms in degrees at most $p$.
The associated long exact sequence therefore gives an injection
\begin{equation}
\label{eq:multi-cech-injection}
R^pf_*\bbQ_B=\mathcal H^p(\mathcal T)
 \lhook\joinrel\longrightarrow\mathcal H^p(\mathcal T_{\leq p}).
\end{equation}
Thus discarding the later columns does not change the kernel of
any map into $R^pf_*\bbQ_B$.  

On $S$, each retained vertical complex has locally constant
cohomology sheaves.  The finite column filtration of
$\mathcal T_{\leq p}$ therefore shows that
$\mathcal H^p(\mathcal T_{\leq p})|_S$ is a local system:
complexes with locally constant cohomology are closed under the
finite cones and extensions used in this filtration.

Put $M=\HH^p(A,\bbQ)$.  The inclusion $B\hookrightarrow X\times A$
and proper base change give a morphism
\[
\underline M_X\longrightarrow R^pf_*\bbQ_B.
\]
Its stalk at $x$ is the restriction map $M\to\HH^p(A_x,\bbQ)$.
By \eqref{eq:multi-cech-injection}, its kernel agrees with the
kernel of the composite
\[
\underline M_X\longrightarrow\mathcal H^p(\mathcal T_{\leq p}).
\]
After restriction to $S$, this is a morphism from a constant local
system to a local system.  Its kernel is a local subsystem of
$\underline M_S$.  On each connected component of $S$, transport
in the ambient constant system identifies this subsystem with a
single fixed subspace of $M$.  This proves the asserted constancy,
and \eqref{eq:multi-cech-count} proves the required bound.
\end{proof}

\begin{proof}[Proof of 
Theorem~\ref{thm:main:general}]
\hide{
    Follows from Theorem~\ref{thm:main:complex:general'} using bounds on 
    the number of connected components of the realizations of $0/1$-patterns of definable families of subsets
    in o-minimal structures as well as over algebraically closed fields 
    (dual version of \eqref{eqn:vcd-bound}).
}
    Theorem~\ref{thm:main:complex:general'} supplies a fixed closed correspondence $[\Gamma;X,V]$ and at most $N\le Cn^{p+q}$ fibers whose pattern realizations have constant restriction kernel on each connected component. Applying 
    Theorem~\ref{thm:vcdensityapp} (respectively, Theorem~\ref{thm:vcdensityapp:om} in the o-minimal case) in its single-family case 
    gives
\[
\sum_{\sigma}b_0(\operatorname{Reali}(\sigma))
\le C'(N+1)^{\dim X}
\le C''n^{(p+q)\dim X}.
\]
The number of distinct kernels is at most this sum, proving 
Theorem~\ref{thm:main:general}.
\end{proof}

\begin{proof}[Proof of Theorem~\ref{thm:informal:general}]
    Follows immediately from Theorem~\ref{thm:main:general}.
\end{proof}

\section{Proofs of Theorems~\ref{thm:epsilon-nets-degree-p} and \ref{thm:fractional-helly-degree-p}}
\label{sec:proof:applications}

\subsection{Proof of Theorem~\ref{thm:epsilon-nets-degree-p} 
}

In the proof of Theorem~\ref{thm:epsilon-nets-degree-p}  
we will use the following lemma.

\begin{lemma}[Top-degree kernel recovery]
\label{lem:unique}
Let $p\geq0$ be an integer, and let $Z_1,\ldots,Z_m$ and $X$
be closed semi-algebraic subsets of $\bbR^N$, with
$\dim Z_i\leq p$ for every $i$.
Put
\[
U=\bigcup_{i=1}^m Z_i,\qquad A=U\cap X,
\qquad K=\ker\bigl(\HH^p(U,\bbQ)\to\HH^p(A,\bbQ)\bigr),
\]
and let $r_i:\HH^p(U,\bbQ)\to\HH^p(Z_i,\bbQ)$ be restriction.
Then
\begin{equation}
\label{eq:individual-kernel-recovery}
r_i(K)=\ker\bigl(\HH^p(Z_i,\bbQ)
                    \longrightarrow\HH^p(Z_i\cap X,\bbQ)\bigr).
\end{equation}
Consequently, $K$ uniquely determines all the individual
restriction kernels, and in particular the set
$\{Z_i:Z_i\in_pX\}$.
\end{lemma}

\begin{proof}
We first record why top-degree restriction is surjective.
If $L$ is a closed semi-algebraic subset of $U$, triangulate
the pair $(U,L)$ compatibly, using a locally finite simplicial
complex when $U$ is unbounded.  Its dimension is at most $p$,
and $L$ corresponds to a subcomplex.  There are no
$(p+1)$-simplices, so inclusion is injective on $H_p$.
The universal coefficient theorem over $\bbQ$ therefore shows
that $\HH^p(U,\bbQ)\to\HH^p(L,\bbQ)$ is surjective.

Fix $i$.  The inclusion of the left side of
\eqref{eq:individual-kernel-recovery} in the right side follows
from the commutative restriction square.
For the reverse inclusion, take
$\alpha\in\HH^p(Z_i,\bbQ)$ that restricts to zero on $A\cap Z_i$.
The Mayer--Vietoris sequence for the cover of $A\cup Z_i$
by the closed semi-algebraic subsets $A$ and $Z_i$ contains
\[
\HH^p(A\cup Z_i,\bbQ)
 \longrightarrow\HH^p(A,\bbQ)\oplus\HH^p(Z_i,\bbQ)
 \longrightarrow\HH^p(A\cap Z_i,\bbQ).
\]
The pair $(0,\alpha)$ lies in the kernel of the second arrow,
so it lifts to a class $\beta\in\HH^p(A\cup Z_i,\bbQ)$.
Since $A\cup Z_i$ is closed in $U$ and $\dim U\leq p$,
the preceding surjectivity lifts $\beta$ to
$\gamma\in\HH^p(U,\bbQ)$.
Then $\gamma|_A=0$ and $r_i(\gamma)=\alpha$, proving the identity.

Finally, the individual restriction is nonzero exactly when
its kernel is a proper subspace of $\HH^p(Z_i,\bbQ)$.
These subspaces are recovered by the fixed maps $r_i$ from
$K$.  Neither general position nor irreducibility was used.
\end{proof}

\hide{
\begin{proof}[Proof of Theorem~\ref{thm:epsilon-nets-degree-p}]
Consider the ordinary range space on the ground set $\mathcal Z$
whose ranges are
\[
R_X=\{Z\in\mathcal Z:Z\in_pX\},\qquad X\in\mathcal X.
\]
For every finite sample $F\subset_m\mathcal Z$,
Lemma~\ref{lem:unique} gives
\[
\card\{F\cap R_X:X\in\mathcal X\}
 \leq\nu^p_{\mathcal X,\mathcal Z}(m).
\]
Choose $s$ with $\rho<s<d$.  By the definition of $\rho$, the
right side is at most $m^s$ for all sufficiently large $m$.
For the remaining sample sizes, the number of traces is at most
$2^m$.  Thus the ordinary shatter function of this range space
is bounded by $C_1m^s$ after increasing $C_1$; the same bound,
with a changed constant, holds for samples of size at most $m$.
If $\mathcal Z$ is finite, its trace counts are uniformly bounded
by $2^{\card(\mathcal Z)}$, so the same polynomial bound and
finite VC-dimension hold directly.  In all cases the range space
has some finite VC-dimension $v$.

Apply Theorem~\ref{thm:eps-net} to its restriction to
$\mathcal Z_0$.  This gives a subset
$\mathcal S\subset\mathcal Z_0$ of size at most
$C_0(v+1)\eps^{-1}\log(2/\eps)$ meeting every range whose
trace has size at least $\eps\card(\mathcal Z_0)$.
This is precisely the required degree-$p$ net.
Absorb $(v+1)/d$ into $C$.
\end{proof}
}
\begin{proof}[Proof of Theorem~\ref{thm:epsilon-nets-degree-p}]
Consider the ordinary range family on the ground set
$\mathcal Z$ given by
\[
\mathcal R=\{R_X:X\in\mathcal X\},
\qquad
R_X=\{Z\in\mathcal Z:Z\in_pX\}.
\]
For every finite $F\subseteq\mathcal Z$ with $|F|=m$,
Lemma~\ref{lem:unique}
gives
\begin{equation}
\label{eq:trace-kernel-bound}
\card(\{F\cap R_X:X\in\mathcal X\})
\leq \nu^p_{\mathcal X,\mathcal Z}(m).
\end{equation}
Indeed, putting $U=\bigcup_{Z\in F}Z$, the kernel of
$\HH^p(U,\Q)\to\HH^p(U\cap X,\Q)$ determines the kernel of
$\HH^p(Z,\Q)\to\HH^p(Z\cap X,\Q)$ for each $Z\in F$,
and hence determines $F\cap R_X$.

Choose a real number $s$ with $\rho<s<d$.
If $\mathcal Z$ is infinite, the definition of $\rho$ implies
$\nu^p_{\mathcal X,\mathcal Z}(m)\leq m^s$ for all sufficiently
large $m$.  For the remaining sizes the number of ordinary
traces is at most $2^m$.  Thus, after increasing a constant $A$,
\begin{equation}
\label{eq:polynomial-growth-bound}
\nu_{\mathcal R}(m)\leq A m^s
\qquad\text{for every admissible }m\geq1.
\end{equation}
This also holds when $\mathcal Z$ is finite, since all its
trace counts are bounded by $2^{|\mathcal Z|}$.
The constant $A\geq1$ depends on the fixed families, $p$, and $s$,
and not on $\mathcal Z_0$.

Apply 
Proposition~\ref{prop:Pach}
to \eqref{eq:polynomial-growth-bound}.
Since $s<d$, there is an $\eps_0>0$, depending only on
$A,s,d$, such that every nonempty finite
$\mathcal Z_0\subseteq\mathcal Z$ has an ordinary $\eps$-net
$\mathcal S\subseteq\mathcal Z_0$ for the traces of
$\mathcal R$ satisfying
\[
|\mathcal S|\leq\frac d\eps\ln\frac1\eps
\qquad(0<\eps<\eps_0).
\]
By the definition of $R_X$, this is precisely a degree-$p$
$\eps$-net.  The threshold is independent of $\mathcal Z_0$;
choosing, for example, $s=(\rho+d)/2$ gives the stated
dependence of $\eps_0$.
\end{proof}

\subsection{Proof of Theorem~\ref{thm:fractional-helly-degree-p}}
\begin{proof}[Proof of Theorem~\ref{thm:fractional-helly-degree-p}]
For $Z\in\mathcal Z$, put
\[
\mathcal X_Z=\{X\in\mathcal X:Z\in_pX\},\qquad
\mathcal F=\{\mathcal X_Z:Z\in\mathcal Z\}\subset2^{\mathcal X}.
\]
The dual patterns of $\mathcal F$ on a finite collection of its
members are exactly the incidence patterns counted in the proof
of Theorem~\ref{thm:epsilon-nets-degree-p}.  Repeated members do
not increase the pattern count.  That proof gives, for some
$s<d$, a bound $C_1m^s$ on all such patterns using at most $m$
test sets.  Therefore
\[
\pi^*_{\mathcal F}(m)=O(m^s)=o(m^d).
\]
For any indexed test family $Z_1,\ldots,Z_n$ and index set $I$,
\[
\bigcap_{i\in I}\mathcal X_{Z_i}\neq\emptyset
\quad\Longleftrightarrow\quad
\exists X\in\mathcal X\ \forall i\in I,\ Z_i\in_pX.
\]
Apply Theorem~\ref{thm:fractional-Helly} to
$\mathcal F$ and this indexed family.  The resulting point of
its ground set is an $X\in\mathcal X$ incident to at least
$\beta n$ test sets.  This is the required conclusion.
\end{proof}

\section{Conclusions}
We have extended the concept of VC-density to a topological context. This work was motivated by a dual aim: to advance a theory of higher-order dependence within model theory and to establish higher-degree topological analogs of results in discrete geometry, such as the existence of $\eps$-nets and fractional Helly-type theorems. We have proven bounds on the generalized VC-density, representing significant generalizations of established bounds across several relevant theories. Additionally, some of our theorems extend beyond classical model theory, applying within the broader framework of scheme-theoretic correspondences. Below, we summarize  the key contributions of our paper.

\subsection{Key contributions}
\begin{enumerate}
    \item
As mentioned earlier, previous results on bounding VC densities of definable families in various structures mostly relied on proving bounds on the number of realized $0/1$-patterns (i.e. bounding the dual density), and the proofs of these bounds are quite different in different structures. For example, the proof in the algebraically closed case in \cite{RBG01} uses a linear algebra argument, while in the 
case of o-minimal theories \cite{Basu10}, the argument is more topological involving infinitesimal tubes and inequalities coming from the Mayer-Vietoris exact sequence. We recover these results
in this paper as a consequence of our main theorems in the special case where $p=0$, but the proof is now direct (bounding the density directly instead of the co-density) and uniform across all these structures.
\item
For $p > 0$, our main results 
Theorem~\ref{thm:informal} and Theorem~\ref{thm:main}
are all new. Even in the case
$p=0$, 
the complex and \'{e}tale versions of Theorem~\ref{thm:main}
are new. 
The \'{e}tale version of Theorem~\ref{thm:main}
in particular, generalizes 
    the result that the VC-codensity of families of hypersufaces of bounded degree $Y = \bbA_k^m$ over an arbitrary field
    $k$ is bounded by $m$. This statement is a consequence of the main result in \cite{RBG01} on bounding $0/1$-patterns of families of 
    hypersurfaces of bounded degrees.  
    The \'{e}tale version of Theorem~\ref{thm:main}
    extends this co-density bound to more general definable families
    on an arbitrary scheme $Y$ (the restriction of properness of $Y$
    is not a severe one in this context as one can always embed a given $Y$  in a proper scheme if needed (for example $\bbA^m$ into $\bbP^m$) and then apply the theorem). 
\item
\hide{
    The bound in 
    Theorem~\ref{thm:main} is
    tight
    in each of the three versions of the theorem. We give constructions in Section~\ref{sec:tightness}.
}
We give constructions in Section~\ref{sec:tightness} showing that the bound in Theorem~\ref{thm:main:general}  is tight 
for all bidegrees in the semialgebraic setting; in the
algebraic settings, the examples establish 
tightness for $(p,q) = (0,1)$ and $(1,1)$ with $\dim X = 1$
in the latter case.
\item
    Finally, we give several applications of our results. First we extend the notion of higher order VC-density to higher degree as well, and use
    our main theorem to prove 
    an upper bound on these densities in each of the theories, with sharpness in the cases described above.
    We also give two examples of how our generalized notion of VC-density might lead to topological generalizations of well known results in discrete geometry -- namely, the existence of higher degree $\eps$-nets (that we define in this paper) of constant size, and a higher degree version of the fractional Helly theorem.
\end{enumerate}

\subsection{Future directions}
The homological perspective we adopt in this paper enables us to approach higher-dimensional incidence problems, where incidence is defined homologically rather than in set-theoretic terms. We do not pursue this line of inquiry further here,  and focus instead on higher-degree generalizations of other applications of VC-density bounds—specifically, the existence of $\eps$-nets and fractional Helly’s theorem. Nevertheless, we believe that homological formulations of incidence problems could open promising research avenues.

From a model-theoretic perspective, there has been prior work (by the authors) on incorporating cohomological concepts into first-order logic \cite{Basu-Patel2}. In \emph{loc. cit.} a notion of `cohomological quantifier elimination' is defined and shown to exist in certain theories. The guiding philosophy is that one should be able to express statements about cohomology classes within the framework of first-order logic and model theory, thereby extending model-theoretic notions like quantifier elimination into this broader context. The higher-degree VC-density definition we propose here aligns with this philosophy, and we aim to develop this approach further.

\bibliographystyle{amsplain}
\bibliography{ref}
 
\end{document}